\documentclass[journal]{IEEEtran}

\usepackage{multirow}
\usepackage{float}
\floatstyle{plain}
\newfloat{bottomfigure}{bp}{lop}
\usepackage{cite}
\ifCLASSINFOpdf
  \usepackage[pdftex]{graphicx}
\else
  \usepackage[dvips]{graphicx}
\fi
\usepackage{amsmath}
\usepackage{amsthm}
\usepackage{bbm}

\usepackage{mathtools}
\usepackage{amssymb}
\usepackage[caption=true]{caption}
\usepackage[font=footnotesize]{subfig}
\usepackage{epsfig}

\usepackage{color}
\usepackage{hyperref}
\usepackage{cite}

\newtheorem{corollary}{Corollary}
\newtheorem{mydef}{Definition}
\newtheorem{theo}{Theorem}
\newtheorem{ass}{Assumption}
\newtheorem{lemma}{Lemma}
\newtheorem{remark}{Remark}

\def\b0{\mbox{{\bf 0}}}

\def\balpha{\mbox{$\boldsymbol{\alpha}$}}
\def\bbeta{\mbox{$\boldsymbol{\eta}$}}
\def\blambda{\mbox{$\boldsymbol{\lambda}$}}
\def\btheta{\mbox{$\boldsymbol{\theta}$}}
\def\bgamma{\mbox{$\boldsymbol{\gamma}$}}
\def\bargamma{\mbox{$\boldsymbol{\bar{\gamma}}$}}
\def\bartheta{\mbox{$\boldsymbol{\bar{\theta}}$}}

\def\bSigma{\mbox{$\boldsymbol{\Sigma}$}}

\def\bPhi{\mathbf{\Phi}}

\def\ba{\mbox{{\bf a}}}

\def\bd{\mbox{{\bf d}}}
\def\be{\mbox{{\bf e}}}
\def\bff{\mbox{{\bf f}}}

\def\bh{\mbox{{\bf h}}}

\def\bs{\mbox{{\bf s}}}
\def\bu{\mbox{{\bf u}}}
\def\bv{\mbox{{\bf v}}}

\def\bx{\mbox{{\bf x}}}
\def\by{\mbox{{\bf y}}}
\def\bz{\mbox{{\bf z}}}

\def\bA{\mbox{{\bf A}}}

\def\bD{\mbox{{\bf D}}}

\def\bF{\mbox{{\bf F}}}

\def\bH{\mbox{{\bf H}}}
\def\bI{\mbox{{\bf I}}}
\def\bJ{\mbox{{\bf J}}}

\def\bL{\mbox{{\bf L}}}

\def\bP{\mbox{{\bf P}}}
\def\bQ{\mbox{{\bf Q}}}

\def\bU{\mbox{{\bf U}}}

\def\bbR{\mathbb{{R}}}

\def\cN{\mbox{${\mathcal{N}}$}}
\def\cP{\mbox{${\mathcal{P}}$}}
\def\cQ{\mbox{${\mathcal{Q}}$}}

\def\dbe{\mbox{$\dot{\mbox{\bf e}}$}}
\def\dbs{\mbox{$\dot{\mbox{\bf s}}$}}
\def\dbx{\mbox{$\dot{\mbox{\bf x}}$}}

\begin{document}
\title{Probability of Stability of Synchronization in Random Networks of Mismatched Oscillators}
\author{Saeed~Manaffam,~\IEEEmembership{Student Member,~IEEE,}
        and~Alireza~Seyedi,~\IEEEmembership{Senior Member,~IEEE}% <-this % stops a space
\thanks{The authors are with the Department
of Electrical Engineering and Computer Science, University of Central Florida, Orlando,
FL, 32816 USA e-mail: saeedmanaffam@knights.ucf.edu and alireza.seyedi@ieee.org.}% <-this % stops a space
}
\maketitle
\date{\today}

%% -------------------------SECTION (I)        ABSTRACT-------------------
\begin{abstract}
The stability of synchronization state in networks of oscillators are studied under the assumption that oscillators and their couplings have slightly mismatched parameters. A generalized master stability function is provided that takes the mismatches into account. Using this master stability function a lower bound on the probability of synchronization is derived for regular and random network models. The probability of stability of synchronization is then used to study the phase transition behavior of the networks. Numerical examples using van der Pol oscillators are used to illustrate the results and verify the validity of the analysis. Moreover, the synchronization trend as a function of statistics of mismatches in the coupling and local dynamics is investigated using this numerical example.
\end{abstract}


%%%%%%%%%%%%%%%%%%%%%%%%%%%%%%%%
%%%%%%-----------------INTRODUCTION
%%%%%%%%%%%%%%%%%%%%%%%%%%%%%%%%
\section{Introduction}
The problem of synchronization in a network of identical oscillators was first introduced by Wiener \cite{Wiener65,Arenas08}. Pursuit of the idea by Winfree in his pioneering work \cite{Winfree67} led to this problem being recognized as being important and relevant in many fields of research including biology, physics, and engineering \cite{Newman10}. More recently, the introduction of the framework of master stability function by Pecora and Carroll \cite{Pecora98}, enabled the investigation of the impacts of network structure and the dynamical properties of individual nodes on the stability of the synchronization state \cite{Pecora98,Motter07}. Following the idea of using master stability function to study the network of oscillators, most efforts have been concentrated on the impact of the topology (structure) of different types of networks on the stability of the synchronization state: In \cite{Pecora98}, the short and long wavelength bifurcation phenomena have been studied on regular networks (lattices). Other works look at linking the topological properties such as minimum, maximum and average node degrees, to the stability of the synchronization state in networks \cite{Motter07,Wang02,Manaffam13,Zhou06}. Due to the interesting properties of small-world networks, which have been introduced in the seminal work of Watts and Strogatz \cite{Watts98,Newman99}, most of the following studies were focused on the small-world and scale-free networks. It has been shown that due to better dynamical flow (efficient communications), the synchronization can be stabilized more easily in small-world networks compared to regular networks \cite{Arenas08,Manaffam13,Wu08}. It has also been shown that the synchronizability of networks improves in homogeneous networks in contrast to heterogeneous ones \cite{Nishikawa03}.

Although the study of the synchronization of networks of the identical nodes appears to be matured, few attempts have been made to study networks with nonidentical nodes or couplings. The experiments reported in \cite{Restrepo04} suggest that in the networks where the oscillator dynamics and their couplings vary slightly from each other, the oscillators can be nearly synchronized. That is, the states converge to the vicinity of a certain trajectory (synchronization manifold). In \cite{Sun09} and \cite{Sorrentino11}, a sensitivity analysis of synchronization have been performed for a network of mismatched oscillators. It has been shown that near-synchronization behavior can occur in a network of mismatched oscillators using master stability function. The general stability of the synchronization in network of  dynamical systems with nonidentical dynamics for each node is studied in \cite{Xiang07} and \cite{Zhao12} using the Lyapunov direct method. In \cite{Acharyya12}, an approximate master stability function is proposed and the coupling strength is optimized to achieve ``best synchronization properties''.

In this paper, we investigate the synchronization of a network of mismatched oscillators with mismatched couplings. Our formulation also allows the consideration of uncertainties in network link weights, thus generalizing \cite{Sorrentino11} in addition to its main contributions. Since in presence of mismatch there is no unique synchronization state in the network, we use the concept of $\varepsilon$-synchronization \cite{Sun09}, where the steady states of the nodes in the network fall into an $\varepsilon$-neighborhood of a certain trajectory (synchronization manifold). We then use a generalized master stability function to study the behavior of the network around the synchronization state. The proposed generalized master stability function bounds the oscillator states to a neighborhood of average synchronization trajectory as a function of Lyapunov exponents of the dynamical network. These Lyapunov exponents, in turn, are related to eigenvalues of the Laplacian matrix of the network. We then provide a probabilistic treatment of synchronization behavior in terms of mismatch parameters for regular and random network models. We calculate probability of stability of synchronization, and use it to investigate phase transitions of the synchronization in the network as the network and node parameters vary. Finally, we verify our analytical results by a numerical example for a network of van der Pol oscillators \cite{Poland94} with mismatched oscillators and couplings.

%%%%%%%%%%%%%%%%%%%%%%%%%%%%%%%%
%%%%%%-----------------NOTATIONS
%%%%%%%%%%%%%%%%%%%%%%%%%%%%%%%%
\section{Notation and Main Variables}
The set of real (column) $n$-vectors is denoted by $\bbR^{n}$ and the set of real $m\times n$ matrices is denoted by $\bbR^{m\times n}$. We refer to the set of non-negative real numbers by $\bbR_{+}$. Matrices and vectors are denoted by capital and lower-case bold letters, respectively. Identity matrix is shown by \bI. The Euclidean ($\mathcal{L}_{2}$) vector norm is represented by $\lVert\cdot\rVert $. When applied to a matrix, $\lVert\cdot\rVert $ denotes the $\mathcal{L}_{2}$ induced matrix norm, $\lVert\bA\rVert =\sqrt{\lambda_{\max}(\bA^{T}\bA)}$. Table \ref{tab:var} summarizes the main variables used.
\begin{table}[!h]
\begin{center}\caption{Main variables}\label{tab:var}
\begin{tabular}{ll}
\hline
\bf{Variable} & \bf{Description}\\
\hline\hline
$\bx_{i}$& State vector of node $i$\\\hline
$\bgamma_{i}$& Parameters vector of node $i$\\\hline
$\btheta_{ji}$& Parameter vector of coupling from node $j$ to node $i$\\\hline
$\mathbf{f}(\bx_{i},\bgamma_{i})$& Dynamics function of node $i$ \\\hline
$\bh(\bx_{j},\bx_{i},\btheta_{ji})$& Coupling function from node $j$ to node $i$ \\\hline
$\bu_{i}$&Input vector for node $i$\\\hline
$\bF_{\bx}$&Jacobian of vector $\bff$ with respect to $\bx$\\\hline
$\bF_{\bgamma}$&Jacobian of vector $\bff$ with respect to $\bgamma$\\\hline
$\bH_{\bx}$&Jacobian of coupling vector $\bh$ with respect to $\bx$\\\hline
$\bH_{\by}$&Jacobian of coupling vector $\bh$ with respect to $\by$\\\hline
$\bH_{\btheta}$&Jacobian of coupling vector $\bh$ with respect to $\btheta$\\\hline
\end{tabular}
\end{center}
\end{table}%

%%%%%%%%%%%%%%%%%%%%%%%%%%%%%%%%
%%%%%%-----------------SYSTEM DESCRIPTION
%%%%%%%%%%%%%%%%%%%%%%%%%%%%%%%%
\section{System Description}
Consider a network of $N$ oscillators, indexed by $\cN = \{1~ \cdots ~N\}$. Assume that the dynamics of each isolated oscillator is governed by 
\begin{align}
\dbx_{i}=\bff(\bx_{i},\bgamma_{i}), \nonumber
\end{align}
where $\bx_{i}\in\bbR^{n}$ and $\bgamma_{i}\in\cP\subseteq \bbR^{p}$ are the state and parameter vectors of local dynamics of node $i$, respectively. $\cP$ denotes the set of possible parameter vectors, and $\bff:\bbR^{n+p}\rightarrow\bbR^{n}$ describes the local dynamics of an isolated node.

The dynamics of coupled oscillators are given as
\begin{eqnarray}
\dbx_{i}&=&\bff(\bx_{i},\bgamma_{i})+\sum_{i,j\in \,\cN}a_{ij}\bh(\bx_{j},\,\bx_{i},\,\btheta_{ij}), \label{eq: NetworkDynamics}
\end{eqnarray}
where $\btheta_{ij}\in\cQ\subseteq\bbR^q$ is the parameter vector of coupling dynamics from node $j$ to node $i$, $\cQ$ denotes the set of possible parameter values for couplings. The adjacency matrix of the network is $\bA=[a_{ij}]$, where $a_{ij}\in\bbR_{+}$ is the weight of the link from node $j$ to node $i$. There is no connection if $a_{ij}=0$. Note that we allow the more general case of directed and wighted networks. Moreover, $\bh:\bbR^{2n+q}\rightarrow\bbR^{n}$ models the coupling from node $j$ to node $i$. We assume that $\bh(\bx,\by,\btheta)$ is Hamiltonian. That is, we assume that $\bH_{\bx}=-\bH_{\by}$, where $\bH_{\bx}$ and $\bH_{\by}$ denotes the Jacobians of $\bh(\bx,\by,\btheta)$ with respect to $\bx$ and $\by$, respectively. This is a very general assumption and encompasses the {\em diffusive coupling} model predominantly used in the literature \cite{Nishikawa03,Sorrentino11,Acharyya12}, where it is assumed that $\bh(\bx_1,\bx_2,[\btheta_1~\btheta_2])=\tilde{\bh}(\bx_1,\btheta_1)-\tilde{\bh}(\bx_2,\btheta_2)$.

Note that this generalized model also incorporates uncertainties in the adjacency matrix of the network, $\bA=[a_{ij}]+[\delta a_{ij}]$, considered in \cite{Sorrentino11}, by absorbing $\delta a_{ij}$ into $\theta_{ij}$, i.e., $\btheta'_{ij}=[\btheta_{ij}^T~\delta a_{ij}]^T$.

\section{Invariant Synchronization Manifold}
Let $\bs$ be a weighted average of the trajectories of all oscillators
\begin{eqnarray}
	\bs&=&\sum_{i\in\,\cN}\alpha_{i}\bx_{i},\label{eq: bs}
\end{eqnarray}
where $\sum_{i\in\,\cN}\alpha_{i}=1$. Define the deviation of the trajectory of oscillator $i$ from \bs~as
\begin{eqnarray}
	\be_{i}&=&\bx_{i}-\bs.\label{eq: error}
\end{eqnarray}
Moreover, let $\bL=[l_{ij}]$ be the Laplacian matrix of the network \cite{Mohar91},
\begin{equation}
	\bL=\mbox{diag}([d^{\scriptsize\mbox{in}}_{1} \cdots d^{\scriptsize\mbox{in}}_{N}])-\bA,\nonumber
\end{equation}
where $d_{i}^{\scriptsize\mbox{in}}=\sum_{j\,\in\,\cN}a_{ij}$ is the in-degree of node $i$. 

\begin{lemma}
	$\bs=\sum_{i\in\,\cN}\alpha_{i}\bx_{i}$ is an invariant synchronization manifold of the network if $\balpha=[\alpha_1\cdots\alpha_N]^T$ is a null vector of $\bL^T$. 
\end{lemma}

\begin{proof} 
	Taking derivative of \eqref{eq: bs} yields
	\begin{eqnarray} \label{TrajectoryS}
		\dbs&=&\sum_{i\,\in\,\cN}\alpha_{i}\dbx_{i}\nonumber\\ 
		&=&\sum_{i\,\in\,\cN}\alpha_{i}\bff(\bs+\be_{i},\,\bargamma+\delta\bgamma_{i})\nonumber\\
		&&+\sum_{i,j\in\,\cN}a_{ij}\alpha_{i}\bh(\bs+\be_{j},\,\bs+\be_{i},\,\bartheta+\delta\btheta_{ij}),
	\end{eqnarray}
	where $\bargamma=\sum_{i\,\in\,\cN}\alpha_{i}\bgamma_{i}$, $\delta\bgamma_{i}=\bgamma_{i}-\bar{\bgamma}$, $\bartheta=\frac{1}{\bar{d}_{\scriptsize\mbox{in}}}\sum_{i,j\,\in\,\cN}\alpha_{i}a_{ij}\btheta_{ij}$, $\delta\btheta_{ij}=\btheta_{ij}-\bar{\btheta}$, and $\bar{d}_{\scriptsize\mbox{in}}=\sum_{i\in\cN}\alpha_i d^{\scriptsize\mbox{in}}_i$ is the weighted average in-degree of the network. Linearization of \eqref{TrajectoryS} around $(\bs,\,\bargamma,\,\bartheta)$ results in
	\begin{eqnarray}
		\dbs&=&\sum_{i\,\in\,\cN}\alpha_{i}\bff(\bs,\,\bargamma)+\bF_{\bgamma}\sum_{i\,\in\cN}\alpha_{i}\delta\bgamma_{i}\nonumber\\
		&&+\sum_{i,\,j\in\,\cN}a_{ij}\alpha_{i}\bh(\bs,\,\bs,\,\bartheta)+\bH_{\bx}\sum_{i,\,j\in\,\cN}a_{ij}\alpha_{i}(\be_{j}-\be_{i})\nonumber\\
		&&+\bH_{\btheta}\sum_{i,\,j\in\,\cN}a_{ij}\alpha_{i}\delta\btheta_{ji},\nonumber
	\end{eqnarray}
	where $\bH_{\bx}$ and $\bH_{\btheta}$ are Jacobians of $\bh$ with respect to its first and third variable, respectively. Recalling that $\sum_{i\in\,\cN}\alpha_{i}=1$, we have
	\begin{eqnarray}
		\dbs&=&\bff(\bs,\,\bargamma)+\bh(\bs,\,\bs,\,\bartheta)\sum_{i\,\in\,\cN}d^{\scriptsize\mbox{in}}_{i}\alpha_{i}\nonumber\\
		&&+\bH_{\bx}\sum_{i,j\in\,\cN}a_{ij}\alpha_{i}(\be_{j}-\be_{i}).\nonumber
	\end{eqnarray}

	For $\bs$ to be an invariant manifold, the last term in the above equation must be zero. This is achieved if $\alpha_{i}$ are chosen to satisfy
	\begin{eqnarray}
		\sum_{i,\,j\in\,\cN}a_{ij}\alpha_{i}(\be_{j}-\be_{i})&=&\sum_{i\,\in\,\cN}\left[\sum_{j\,\in\,\cN} (a_{ji}\alpha_{j}-a_{ij}\alpha_{i})\right] \be_{i}\nonumber\\
		&=&\b0.\label{AlphaEq}
	\end{eqnarray}
	Equation \eqref{AlphaEq}, in turn, will be satisfied if $\sum_{j\,\in\,\cN} (a_{ij}\alpha_{j}-a_{ij}\alpha_{i})=0$ for all $i\,\in\,\cN$, which in matrix form can be represented as
	\begin{equation}
		\bA^T\balpha=\mbox{diag}([d^{\scriptsize\mbox{in}}_{1} \cdots d^{\scriptsize\mbox{in}}_{N}])\balpha,\nonumber
	\end{equation}
	where $\balpha=[\alpha_{1}\cdots\alpha_{N}]$, or
	\begin{equation}
		\left[\bA^T-\mbox{diag}([d^{\scriptsize\mbox{in}}_{1} \cdots d^{\scriptsize\mbox{in}}_{N}])\right]\balpha=\b0=\bL^T\balpha.\nonumber
	\end{equation}
	That is, $\balpha$ is a null vector of $\bL^T$. 
\end{proof}

\begin{remark} 
We note that, by definition, $\bL$ has zero row sum. Thus, it is singular. Consequently, $\bL^T$ always has a null vector, $\balpha$. This means that any network has at least one invariant manifold. 
\end{remark}

\begin{remark} 
If the the network is connected, the invariant synchronization manifold is unique. This is due to the fact that for connected networks the nullity of $\bL$ is one. Thus, $\balpha$ and, therefore, $\bs$ are unique.
\end{remark}

\begin{remark} In the special case where the network is undirected, $\bL$ is symmetric. Thus, it also has zero column-sum. Consequently, $\balpha = \frac{1}{N}[1~ \cdots~1]$ is its null vector, and the invariant manifold, $\bs$, is the simple average of the trajectories.
\end{remark}

With $\alpha_{i}$ chosen such that $\bs$ is an invariant manifold, we have
\begin{eqnarray}\label{eq: Trajectory}
	 \dbs&=&\bff(\bs,\,\bargamma)+\bar{d}_{\scriptsize\mbox{in}}\,\bh(\bs,\,\bs,\,\bartheta),\\
	 \bs({0})&=&\sum_{i\,\in\,\cN}\alpha_i\bx_{i}({0}),\nonumber
\end{eqnarray}
where $\bs(0)$  and $\bx({0})$ are initial states. 

%%%%%%%%%%%%%%%%%%%%%%%%%%%%%%%%
%%%%%%-----------------STABILITY ANALYSIS
%%%%%%%%%%%%%%%%%%%%%%%%%%%%%%%%
\section{Generalized Master Stability Function}
In this section we introduce a master stability function which generalizes those in \cite{Sun09} and \cite{Sorrentino11} by taking into account the parameter mismatch in the links and applies to directed and weighted networks.

As it has been shown in previous section, every connected network has a unique invariant manifold. Hence, we can define $\varepsilon$-synchronization as
\begin{mydef} A network of oscillators is $\varepsilon$-synchronized if there exists $\varepsilon>0$ such that
	\begin{align}
		\limsup_{t\to\infty} \lVert\be\rVert \le \varepsilon,\nonumber
	\end{align}
where $\be = [\be_1~...~\be_N]^T$. 
\end{mydef}
This definition means that the error from the manifold is contained in a ball of radius $\varepsilon$.
We note that our definition is different but closely related to that given in \cite{Zhao12}.

Substituting \eqref{eq: NetworkDynamics} and \eqref{eq: bs} in \eqref{eq: error}, and using Taylor series, the dynamics of the error with respect to the synchronization manifold, $\be_i$, is given by
\begin{eqnarray}\label{eq: ErrorDynamics}
\dbe_{i}&=&\bF_{\bx}\be_{i}-\sum_{j=1}^{N}l_{ij}\bH_{\bx}\be_{j}+ \bF_{\bgamma}\delta\bgamma_{i}+\sum_{j=1}^{N}\mathbbm{1}_{i\ne j}l_{ij}\bH_{\btheta} \delta\btheta_{ij}\nonumber\\
&&+(d_i^{\scriptsize\mbox{in}}-\bar{d}_{\scriptsize\mbox{in}})\bh(\bs,\,\bs,\,\bartheta),
\end{eqnarray}
where $\mathbbm{1}_{X}$ is the indicator function of $X$. Stacking (\ref{eq: ErrorDynamics}) for all $i$ yields the dynamics of the deviation of node trajectories from $\bs$:
\begin{eqnarray}\label{eq: ErrorDynamicsNonDiag}
\dot{\be}&=&\left(\bI\otimes\bF_{\bx}-\bL\otimes \bH_{\bx}\right)\delta\bx+\left(\bI\otimes\bF_{\bgamma}\right)\delta\bgamma\nonumber\\
&&+\left(\boldsymbol{\mathcal{A}}\otimes\bH_{\btheta}\right)\delta \btheta+(\bd^{\scriptsize\mbox{in}}-\bar{d}_{\scriptsize\mbox{in}}\mathbf{1}^T_N)\otimes\bh(\bs,\,\bs,\,\bartheta),
\end{eqnarray}
where 
\begin{eqnarray}
{\delta \bgamma}&=&[\delta\bgamma_{1}^{T} \cdots \delta\bgamma_{N}^{T}]^{T},\nonumber\\
{\delta \btheta}& = & [\delta\btheta_{11}^{T}\cdots\delta\btheta_{1N}^{T}~\delta\btheta_{21}^T\cdots\delta\btheta_{2N}^T~\cdots~\btheta_{N1}^T\cdots\delta\btheta_{NN}^T]^{T},\nonumber\\
\bd^{\scriptsize\mbox{in}}&=&[d_1^{\scriptsize\mbox{in}} \cdots d_N^{\scriptsize\mbox{in}}]^T,\nonumber\\
\boldsymbol{\mathcal{A}}&=&\mbox{diag}([\ba_1 \cdots \ba_N]),\nonumber
\end{eqnarray}
and $\ba_i$ is the $i$th row of $\bA$.

Let $\bL=\bP\bJ\bP^{-1}$ be the Jordan decomposition of $\bL$, where $\bP=[p_{ij}]$ is a similarity transform and $\bJ$ is in Jordan form. Then, \eqref{eq: ErrorDynamicsNonDiag} can be rewritten as
\begin{eqnarray}\label{eq: NetDiag1}
\dot{\be}&=&\left(\bP\otimes\bI\right)\left(\bI\otimes\bF_{\bx}-\bJ\otimes \bH_{\bx}\right)\left(\bP^{-1}\otimes\bI\right)\be+\left(\bI\otimes\bF_{\bgamma}\right){\delta \bgamma}\nonumber\\
&&+\left(\boldsymbol{\mathcal{A}}\otimes\bH_{\btheta}\right){\delta \btheta}+(\bd^{\scriptsize\mbox{in}}-\bar{d}_{\scriptsize\mbox{in}}\mathbf{1}_N^T)\otimes\bh(\bs,\,\bs,\,\bartheta).\nonumber
\end{eqnarray}
Using the similarity transform
 \begin{equation}
 \bbeta=\left(\bP^{-1}\otimes\bI\right)\be,\nonumber
 \end{equation}
 where $\bbeta=[\bbeta^T_1\cdots\bbeta_N^T]^T$, we obtain
\begin{eqnarray}\label{eq: NetDiag2}
\dot{\bbeta}&=&\left(\bI\otimes\bF_{\bx}-\bJ\otimes \bH_{\bx}\right)\bbeta+\left(\bP^{-1}\otimes\bI\right)\left(\bI\otimes\bF_{\bgamma}\right){\delta \bgamma}\nonumber\\
&&+\left(\bP^{-1}\otimes\bI\right)\left(\boldsymbol{\mathcal{A}}\otimes\bH_{\btheta}\right){\delta \btheta}\nonumber\\
&&+\left(\bP^{-1}\otimes\bI\right)\left((\bd-\bar{d}_{\scriptsize\mbox{in}}\mathbf{1}_N^T)\otimes\bh(\bs,\,\bs,\,\bartheta)\right)\nonumber\\
&=&\left(\bI\otimes\bF_{\bx}-\bJ\otimes\bH_{\bx}\right)\bbeta+\left(\bP^{-1}\otimes\bF_{\bgamma}\right){\delta \bgamma}\nonumber\\
&&+\left(\bP^{-1}\boldsymbol{\mathcal{A}}\otimes\bH_{\btheta}\right){\delta \btheta}\nonumber\\
&&+\left(\bP^{-1}(\bd-\bar{d}_{\scriptsize\mbox{in}}\mathbf{1}^T_N)\right)\otimes\bh(\bs,\,\bs,\,\bartheta).\nonumber\\
&=&\left(\bI\otimes\bF_{\bx}-\bJ\otimes\bH_{\bx}\right)\bbeta+\bv,
\end{eqnarray}
where $\bv=[\bv_1\cdots\bv_N]$,
\begin{eqnarray}
\bv_{i}&=&\sum_{j\,\in\,\cN}q_{ij}\Bigg[\bF_{\bgamma}\delta \bgamma_{j}+\sum_{k=1,k\ne j}^{N} a_{jk}\bH_{\btheta}\delta\btheta_{jk}\nonumber\\
&&+\bh(\bs,\,\bs,\,\bartheta)(d_i^{\scriptsize\mbox{in}}-\bar{d}_{\scriptsize\mbox{in}})\Bigg],\nonumber
\end{eqnarray}
and $q_{ij}$ are the elements of $\bQ=\bP^{-1}$. It is clear that stability of $\bbeta$ and $\be$ are equivalent. 

To study the stability of \eqref{eq: NetDiag2}, let us first consider the simpler case where $\bJ$ consists of a single Jordan block, i.e. 
\begin{eqnarray}
\bJ=\bJ_{N}(\mu)=\left[\begin{array}{cccccc}
\mu& 1&0&\cdots&0& 0\\
0&\mu&1&\cdots & 0 & 0\\
0&0 &\mu&\cdots & 0 & 0\\
\vdots&\vdots&\vdots&\ddots&\vdots&\vdots\\
0&0&0&\cdots&\mu & 1\\
0&0&0&\cdots&0 & \mu
 \end{array}\right].\nonumber
\end{eqnarray}

\begin{lemma}\label{boundlemma}
For system
\begin{eqnarray*}
\dot{\bbeta}=\left(\bI\otimes\bF_{\bx}-\bJ_N(\mu)\otimes\bH_{\bx}\right)\bbeta+\bv,
\end{eqnarray*}
there exists $\phi>0$ such that
\begin{eqnarray}
\limsup_{t\to\infty}\lVert\bbeta_i\rVert\le\sum_{j=i}^N \left(\frac{\phi}{\lambda}\right)^{N-j+1}\limsup_{t\to\infty}\lVert\bF_{\bx}-\bH_{\bx}\rVert^{N-j} \nonumber\\
\times \limsup_{t\to\infty}\lVert\bv_j\rVert,\nonumber
\end{eqnarray}
for all $i$, if $\lambda>0$, where $\lambda=\mbox{MLE}(\bF_{\bx}-\mu\bH_{\bx})$, and $\mbox{MLE}(.)$ returns the maximum Lyapunov exponent of the argument.
\end{lemma}

\begin{proof} 
See Appendix \ref{NonDiag}.
\end{proof}

Now, let us assume that $\bJ$ consists of $M$ Jordan blocks with eigenvalues $\mu_m$ and sizes $n_m,~m\in\{1,\cdots,M\}$, where $\sum_{m=1}^M n_m=N$. Then $N_m=\sum_{m=1}^j n_m$ will be the index of the last row of the $m$th Jordan block. Define $J(i)$ to be the index of the Jordan block that contains the $i$th row of $\bJ$. In other words, $J(i)=m$, if $N_{m-1}<i\le N_m$.

\begin{theo}\label{Theo-E-bound}
A network of oscillators is $\varepsilon$-synchronized if $\lambda_m=\mbox{MLE}(\bF_{\bx}-\mu_m\bH_{\bx})>0$ and
\begin{align}
\|\bP\|^2\sum_{j=1}^{N-1}\Bigg(\sum_{k=1}^{n_{J(j)}} \left(\frac{\phi_{J(j)}}{\lambda_{J(j)}}\right)^{n_{J(j)}-k+1} \quad\quad\quad\quad\quad\quad\quad\nonumber\\
\times \limsup_{t\to\infty}\lVert\bF_{\bx}-\bH_{\bx}\rVert^{n_{\scriptsize J(j)}-k}\limsup_{t\to\infty}\lVert\bv_{j}\rVert\bigg)^2  \leq{\varepsilon}^2,\label{eq: theo1}
\end{align}
where $\phi_m$ satisfies
\begin{eqnarray}
\forall t,\tau,\quad\|\bPhi_m(t,\tau)\| \le \phi_me^{-\lambda_m(t-\tau)}.\nonumber
\end{eqnarray}
and $\bPhi_m(t,\tau)$ is the state transition matrix of $\bF_{\bx}-\mu_m\bH_{\bx}$.
\end{theo}
\begin{proof} We have
\begin{eqnarray}
\limsup_{t\rightarrow\infty}\lVert\be\rVert^2 & = & \limsup_{t\to\infty}\bbeta^T(\bP^T\otimes\bI)(\bP\otimes\bI)\bbeta\nonumber\\
&\le & \|\bP^T\bP\| \limsup_{t\to\infty}\|\bbeta\|^2\nonumber\\
&=&\|\bP\|^2\limsup_{t\to\infty} \|\bbeta\|^2.\label{eq: GeneralInequality}
 \end{eqnarray}
For any Laplacian matrix, we have $\mu_M=0$. If the network is connected, we  further have $n_M=1$. Thus, 
\begin{eqnarray}
\bbeta_N=\sum_{j=1}^N \alpha_j\be_j=\sum_{j=1}^N \alpha_j(\bx_j-\bs)=\left(\sum_{j=1}^N \alpha_j\bx_j\right) -\bs=\b0,\nonumber
\end{eqnarray}
which together with \eqref{eq: GeneralInequality} yields
\begin{eqnarray}
\limsup_{t\rightarrow\infty}\lVert\be\rVert^2 & \le & \|\bP\|^2\sum_{i=1}^{N-1}\limsup_{t\to\infty}\|\bbeta_i\|^2.\label{eq: inequalityTheo1}
\end{eqnarray}
Lemma \ref{boundlemma} upper bounds the right hand side of \eqref{eq: inequalityTheo1} by the left hand side of \eqref{eq: theo1}.
\end{proof}

\begin{corollary}\label{LemmaSymmetricBound}
A symmetric network of oscillators is $\varepsilon$-synchronized if $\lambda_j=\mbox{MLE}(\bF_{\bx}-\mu_j\bH_{\bx})>0$ and
\begin{equation}\label{eq: coro1}
\sum_{j=1}^{N-1}\left(\frac {\phi_j}{\lambda_j}\right)^2\limsup_{t\rightarrow\infty}\lVert\bv_j(t)\rVert^2 \leq{\varepsilon^2}, %~~~~\forall i\in\{1,\cdots,N-1\}
\end{equation}
where $\phi_j$ satisfies
\begin{eqnarray}
\forall t,\tau,\quad\|\bPhi_j(t,\tau)\| \le \phi_je^{-\lambda_j(t-\tau)},\nonumber
\end{eqnarray}
and $\bPhi_j(t,\tau)$ is the state transition matrix of $\bF_{\bx}-\mu_j\bH_{\bx}$.
\end{corollary}

\begin{proof} 
Since $\bL$ is symmetric, it can be diagonalized by unitary matrix $\bP=\bU=[u_{ij}]$, where $\bU^H\bU=\bI$. Thus, each Jordan block will be of size 1. This means that $M=N$, $n_m=1$, and $J(i)=i$. Thus, \eqref{eq: theo1} reduces to
	\begin{align}
		\sum_{j=1}^{N-1}\left(\frac{\phi_{j}}{\lambda_{j}}\right)^2 \limsup_{t\to\infty}\lVert\bv_{j} \rVert^2 \leq{\varepsilon}^2.\nonumber
	\end{align}
\end{proof}
	\begin{remark}
	In proof of Corollary \ref{LemmaSymmetricBound}, since unitary transformation preserves Euclidean norm, \eqref{eq: inequalityTheo1} holds with equality. Thus, Corollary \ref{LemmaSymmetricBound} is relatively less conservative than Theorem \ref{Theo-E-bound}.
	\end{remark}
 % ---------------------SECTION B: PROBABILITY OF STABILITY
 \section{Probability of Stability}
 In the remaining of the paper, we make the following assumptions:
\begin{ass} \label{ass:Symmetric}
The network is symmetric.
\end{ass}
This implies that $\bL$ is diagonalizable by a unitary matrix, $\bU=[u_{ij}]$. 

\begin{ass} \label{ass:Gaussian}
Mismatch parameters, $\delta \bgamma_{i}$ and $\delta\btheta_{ij}$, are independent zero mean Gaussian random vectors with covariance matrices, $\boldsymbol{\Sigma}_{\bgamma}=E[(\bgamma_{i}-\bargamma)(\bgamma_{i}-\bargamma)^{T}]$ and $\boldsymbol{\Sigma}_{\btheta}=E[(\btheta_{ij}-\bartheta)(\btheta_{ij}-\bartheta)^{T}]$, respectively. 
\end{ass}

Under Assumption \ref{ass:Gaussian}, $\bv_i$ are linear combination of independent Gaussian random variables. Thus, they are jointly Gaussian. To calculate the probability of \eqref{eq: coro1} being satisfied, we need to find the probability density function of $\bv=[\bv_1^T \cdots \bv_{N}^T]^T$.

\begin{lemma}\label{lemma: Covariance}
The covariance matrix of $\bv$ is $\bSigma_{\bv}=[\bSigma_{ij}]$ where
\begin{align}
	\bSigma_{ij}=\bF_{\bgamma} \bSigma_{\bgamma} \bF_{\bgamma}^T \mathbbm{1}_{i=j}+\sum_{l=1}^N u_{il}u_{jl}^* \sum_{k=1}^N a^2_{lk} \bH_{\btheta}  \bSigma_{\btheta} \bH^T_{\btheta}.\nonumber
\end{align}
\end{lemma}

\textit{Proof:} See Appendix \ref{covariance}.

We now provide upper bounds on the probability of stable synchronization for unweighted regular, Erd\"os-R\'enyi, and Newman-Watts networks.

\begin{theo}\label{TheoremRing}
Under Assumptions \ref{ass:Symmetric} and \ref{ass:Gaussian},  the probability of stable synchronization of an unweighted $K$-regular network of oscillators is lower bounded by 
\begin{align}\label{eq: Pstab}
P_{\mbox{\scriptsize stab}}^{\mbox{\scriptsize LB}}(\varepsilon)=&\left[\prod_{i=2}^{N-1}\left(\frac{\phi_1~\lambda_i}{\lambda_1~\phi_i}\right)^{n}\right]\nonumber\\
&\times \sum_{j=0}^\infty a^{(N-1)}_jP\left(\frac{(N-1) n}{2}+j,\frac{\lambda_1^2\varepsilon^2}{\phi_1^2\sigma^2}\right),
\end{align}
where $P(\cdot,\cdot)$ is the regularized gamma function,
\begin{eqnarray*}
a_j^{(i)} & = & \sum_{k=0}^j a_k^{(i-1)}\frac{n_{j-k}}{(j-k)!}(1-\frac{\phi_1^2~\lambda_i^2}{\lambda_1^2~\phi_i^2})^{j-k},\\
a_k^{(2)} & = & \frac{n_k}{k!}\left(1-\frac{\phi_1^2~\lambda_2^2}{\lambda_1^2~\phi_2^2}\right)^k,\\
n_k & = & \prod_{l=0}^{k-1} \left(\frac{n}{2}+l\right),
\end{eqnarray*}
and 
	\begin{align}
		\sigma=\limsup_{t\to\infty}\| \bF_{\bgamma} \Sigma_{\bgamma} \bF_{\bgamma}^T+ K\bH_{\btheta}\Sigma_{\btheta}\bH_{\btheta}^T\|^{1/2}.\label{eq: normSigma}
	\end{align}
\end{theo}

\begin{proof} 
Since the network is unweighted and $K$-regular, we have $d_i=\sum_{k=1}^Na^2_{ik}=K$. According to Lemma \ref{lemma: Covariance} the blocks of the covariance matrix, $\bSigma_{\bv}$, are
\begin{eqnarray}
\boldsymbol{\Sigma}_{{i}{j}} & =& \bF_{\bgamma}{\bSigma}_{\bgamma}\bF_{\bgamma}^{T}\mathbbm{1}_{i=j}+K\bH_{\btheta}~{\bSigma}_{\btheta}\bH^{T}_{\btheta}\sum_{l=1}^N u_{il}u_{jl}^*.\nonumber\\
& =& (\bF_{\bgamma}{\bSigma}_{\bgamma}\bF_{\bgamma}^{T}+K\bH_{\btheta}{\bSigma}_{\btheta}\bH^{T}_{\btheta})\mathbbm{1}_{i=j}.\label{eq: Cov. Regular}
\end{eqnarray}
Hence, $\bv_i$ are uncorrelated. The mean value of $\bv_i$ can be computed as
\begin{eqnarray}
E[\bv_{i}]&=&\sum_{j\,\in\,\cN}q_{ij}\Bigg(\bF_{\bgamma}E[\delta \bgamma_{j}]+\sum_{k=1,k\ne j}^{N} a_{jk}\bH_{\btheta}E[\delta\btheta_{jk}]\nonumber\\
&&+\bh(\bs,\,\bs,\,\bartheta)(K-\bar{d}_{\scriptsize\mbox{in}})\Bigg)=\mathbf{0},\nonumber
\end{eqnarray}
which follows noting that $\delta\bgamma_{i}$ and $\delta\btheta_{ij}$ have zero mean and $\bar{d}_{\scriptsize\mbox{in}}=K$. Since $\bv_i$ are jointly Gaussian, uncorrelated, and have zero mean, they are independent.

Now, let us define the whitened Gaussian random vectors
\begin{align}
\bz_i=\mathbf{\boldsymbol{\Sigma}}_{ii}^{-\frac{1}{2}}\bv_{i}.\nonumber
\end{align}
Since Euclidean norm is sub-multiplicative, we have
\begin{eqnarray}\label{eq: supB}
\lVert\bv_{i}\rVert & \leq & \lVert\bz_{i}\rVert \left\|\mathbf{\boldsymbol{\Sigma}}_{ii}^{\frac{1}{2}}\right\|.\nonumber\\
\limsup_{t\to\infty}\lVert\bv_{i}\rVert & \leq & \limsup_{t\to \infty} \left(\lVert\bz_{i}\rVert \left\|\mathbf{\boldsymbol{\Sigma}}_{ii}^{\frac{1}{2}}\right\|\right).\nonumber\\
& \leq & \limsup_{t\to \infty} \lVert\bz_{i}\rVert \limsup_{t\to \infty}\left\|\mathbf{\boldsymbol{\Sigma}}_{ii}^{\frac{1}{2}}\right\|.\nonumber\\
& = & \lVert\bz_{i}\rVert \limsup_{t\to \infty}\left\|\mathbf{\boldsymbol{\Sigma}}_{ii}^{\frac{1}{2}}\right\|,\nonumber
\end{eqnarray}
The last equality is due to the fact that with the whitening of $\|\bv_i\|$, $\|\bz_i\|$ is no longer time variable. In other words, $\|\bz_i\|$ is a random variable (not a random process). Since $\lVert\bz_{i}\rVert^2$ is the norm squared of a white Gaussian $n$-vector, it has a chi-squared distribution with $n$ degrees of freedom. Applying the result of Corollary \ref{LemmaSymmetricBound}, 
\begin{eqnarray}
\lVert\be\rVert^2&=&\sum_{i=1}^{N-1}\lVert\bbeta_i\rVert^2\nonumber\\
&\le&\sum_{i=1}^{N-1}\left(\frac{\phi_i}{\lambda_i}\right)^2\limsup_{t\to\infty}\lVert\bv_i\rVert^2\nonumber\\
&\le& \sum_{i=1}^{N-1}\left(\frac{\phi_i\sigma}{\lambda_i}\right)^2\lVert\bz_i\rVert^2.\nonumber
\end{eqnarray}
where $\sigma$ is defined in \eqref{eq: normSigma}. 

Now we have
\begin{eqnarray}
\mbox{Pr}\left(\limsup_{t\to\infty}\|\be\|<\varepsilon\right) & = & \mbox{Pr}\left(\limsup_{t\to\infty}\|\be\|^2<\varepsilon^2\right)\nonumber\\
& \ge & \mbox{Pr}\left(\sum_{i=1}^{N-1}\left(\frac{\phi_i\sigma}{\lambda_i}\right)^2\lVert\bz_i\rVert^2\le\varepsilon^2\right)\nonumber\\
& = & \left[\prod_{i=2}^{N-1}\left(\frac{\phi_1~\lambda_i}{\lambda_1~\phi_i}\right)^{n}\right]\nonumber\\
&&\times \sum_{j=0}^\infty a^{(N-1)}_j\int_0^{\varepsilon^2}f_j(y)dy,\nonumber
\end{eqnarray}
where
\[f_j(y)=\left(\frac{\lambda_1^2}{2\phi_1^2\sigma^2}\right)^{\frac{(N-1)n}{2}+j}\frac{y^{\frac{(N-1) n}{2}+j-1}}{\Gamma\left(\frac{(N-1)n}{2}+j\right)}e^{-\frac{\lambda_1^2}{2\phi_1^2\sigma^2}y}.
\]
which using the results in \cite{Moschopoulos84} yields \eqref{eq: Pstab}.
\end{proof}

\begin{theo}\label{TheoremER}
Under Assumptions \ref{ass:Symmetric} and \ref{ass:Gaussian},  the limiting probability of stable synchronization of an unweighted Erd\"os-R\'enyi (ER) network of oscillators, with parameter $p$, as $N\to\infty$, is lower bounded by
\begin{align}\label{eq: Pstab2}
P_{\mbox{\scriptsize stab}}^{\mbox{\scriptsize LB}}(\varepsilon|\blambda)=&\left[\prod_{i=2}^{N-1}\left(\frac{\phi_1~\lambda_i}{\lambda_1~\phi_i}\right)^{n}\right]\nonumber\\
&\times\sum_{j=0}^\infty a^{(N-1)}_jP\left(\frac{(N-1) n}{2}+j,\frac{\lambda_1^2\varepsilon^2}{\phi_1^2\sigma^2}\right),
\end{align}
where $\sigma=\limsup_{t\to\infty}~\|\bF_{\bgamma} \bSigma_{\bgamma} \bF_{\bgamma}^T+pN \bH_{\btheta}  \bSigma_{\btheta} \bH^T_{\btheta}\|^{1/2}$ and $\blambda=[\lambda_1 \cdots \lambda_{N-1}]$.
\end{theo}

\begin{proof} The largest eigenvalue of the Laplacian matrix of any symmetric network is bounded below by the maximum degree of the network. For large ER networks ($N\to\infty$), it is also bounded above by $Np+\sqrt{Np(1-p)}$ \cite{Manaffam13}. Thus
\[d_{\max}\le\mu_{\max}\le Np+\sqrt{Np(1-p)}.\]
Similarly, the smallest non-zero eigenvalue of ER network can be bounded as 
\[d_{\min}\ge\mu_{\min}\ge Np-\sqrt{Np(1-p)}.\]
According to Lemma \ref{lemma: Covariance}, the diagonal blocks of covariance matrix of $\bv$ are 
\begin{eqnarray}
\bSigma_{ii} & = & \bF_{\bgamma} \bSigma_{\bgamma} \bF_{\bgamma}^T+\bH_{\btheta}  \bSigma_{\btheta} \bH^T_{\btheta} \sum_{l=1}^N|u_{il}|^2 d_l\nonumber\\
& = & \bF_{\bgamma} \bSigma_{\bgamma} \bF_{\bgamma}^T+Np\bH_{\btheta}  \bSigma_{\btheta} \bH^T_{\btheta},\nonumber
\end{eqnarray}
as $N\to\infty$, and the off diagonal entries are
\begin{align*}
\bSigma_{ij}&=\bH_{\btheta}  \bSigma_{\btheta} \bH^T_{\btheta} \sum_{l=1}^N u_{il}u_{jl}^* d_l\\
	        &=\bH_{\btheta}  \bSigma_{\btheta} \bH^T_{\btheta}\sum_{l=1}^N u_{il}u_{jl}^* (d_l-Np)\\
	        &\le\sqrt{Np(1-p)}\bH_{\btheta}  \bSigma_{\btheta} \bH^T_{\btheta}.
\end{align*}
Therefore,
\[\lim_{N\to\infty}\frac{\lVert\bSigma_{ij}\rVert}{\lVert\bSigma_{ii}\rVert}=0.\]
Consequently, as $N\to\infty$, $\bv_i$ become independent. The remaining of the proof is similar to that of Theorem \ref{TheoremRing} and is omitted in the interest of brevity.
\end{proof}

To study the synchronization in small-world networks, we consider the Newman-Watts model \cite{Newman99}. This model constructs a small-world network by starting from a $K$-regular ring network (Fig. \ref{fig: Ring}) as substrate, then randomly adds new links with probability $p$.
\begin{figure}[!t]
\begin{center}
\includegraphics[width=1.5in]{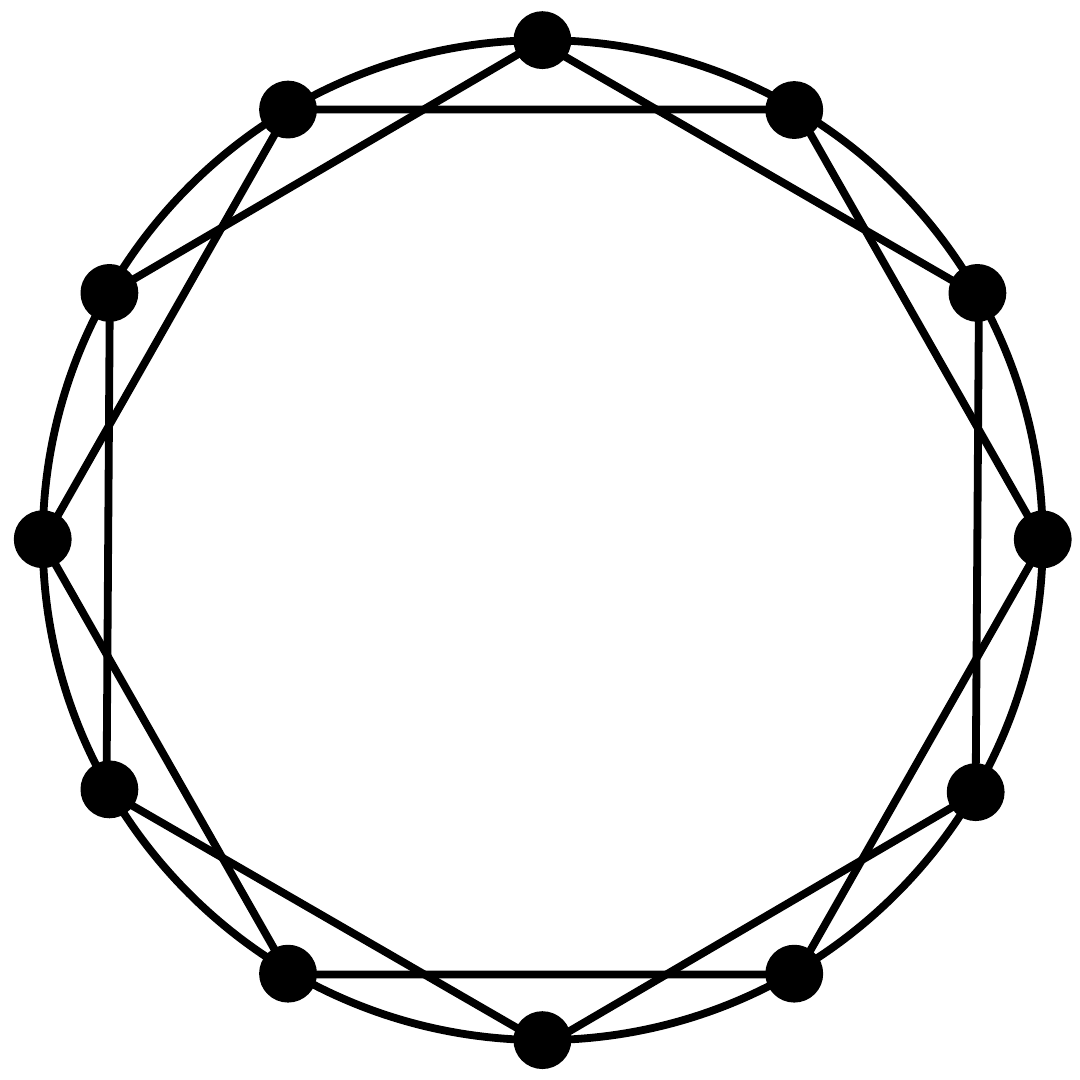}
\caption{A 4-regular ring network.}
\label{fig: Ring}
\end{center}
\end{figure}

\begin{theo}
Under assumptions \ref{ass:Symmetric} and \ref{ass:Gaussian},  the limiting probability of stable synchronization of an unweighted Newman-Watts small-world network of oscillators, with parameters $p$ and $K$, as $N\to\infty$, is lower bounded by
 \begin{align}\label{eq: Pstab2}
 P_{\mbox{\scriptsize stab}}^{\mbox{\scriptsize LB}}(\varepsilon|\blambda)=&\left[\prod_{i=2}^{N-1}\left(\frac{\phi_1~\lambda_i}{\lambda_1~\phi_i}\right)^{n}\right]\nonumber\\
&\times\sum_{j=0}^\infty a^{(N-1)}_jP\left(\frac{(N-1) n}{2}+j,\frac{\lambda_1^2\varepsilon^2}{\phi_1^2\sigma^2}\right),
 \end{align}
 where 
\begin{eqnarray}
\sigma=\limsup_{t\to\infty}~\|\bF_{\bgamma} \bSigma_{\bgamma} \bF_{\bgamma}^T+(K+Np) \bH_{\btheta}  \bSigma_{\btheta} \bH^T_{\btheta}\|^{1/2}.\nonumber
\end{eqnarray}
\end{theo}

\begin{proof} The Laplacian matrix of a Newman-Watts small world network is
\[\bL_{\mbox{\scriptsize NW}}=\bL_{\mbox{\scriptsize Ring}}+\bL_{\mbox{\scriptsize ER}},\]
where $\bL_{\mbox{\scriptsize Ring}}$  and $\bL_{\mbox{\scriptsize ER}}$ are the laplacians of a $K$-regular ring and an Erd\"os-R\'enyi network with parameter $p$. Using Weyl's inequalities we can bound the minimum and maximum eigenvalues of the small-world \cite{Mohar91}
\begin{eqnarray*}
\max\{\mu^{\mbox{\scriptsize Ring}}_{\min},\mu^{\mbox{\scriptsize ER}}_{\min}\}\le\mu_{\min}^{\mbox{\scriptsize NW}}\le d_{\min},\\
\mu^{\mbox{\scriptsize Ring}}_{\max}+\mu^{\mbox{\scriptsize ER}}_{\max}\ge\mu_{\max}^{\mbox{\scriptsize NW}}\ge d_{\max},
\end{eqnarray*}
where the eigenvalues of a $K$-regular ring is \cite{Mohar91}
\begin{eqnarray}
\mu^{\mbox{\scriptsize Ring}}_i & = & K-2\frac{\sin\frac{iK\pi}{2N}\cos\frac{(K+2)i\pi}{2N}}{\sin\frac{i\pi}{N}},\nonumber
\end{eqnarray}
and subscripts $\min$ and $\max$ refer to smallest non-zero and maximum eigenvalue of $\bL$ in corresponding configurations, respectively. The remaining of the proof is similar to that of Theorem \ref{TheoremER} and is omitted in the interest of brevity.
\end{proof}

\section{Numerical Example}
In this section we verify our analytical results using numerical examples. We consider the van der Pol oscillator \cite{Poland94} which has the following dynamics 
\begin{align}\label{eq: osc}
\bff(\bx_{i},\bgamma_{i})=\left[\begin{array}{c}x_{i2}\\
-x_{1i}-\gamma_i (x^2_{i1}-1)x_{i2}\end{array}\right].\nonumber
\end{align}
We note that since the van der Pol oscillator has a limit cycle, as $t\to\infty$, $\bs$ is a periodic trajectory. Hence, the Jacobians are also periodic. We can, therefore, solve \eqref{eq: Trajectory} analytically using Fourier series \cite{Poland94}.

We assume that the nodes are coupled through their first states by 
\[\bh(\bx_{j},\bx_{i},\btheta_{ij})=\left[\begin{array}{c}\theta_{ij1}(x_{1j}-x_{1i})+\theta_{ij2}\\0\end{array}\right].\] 
Thus, the Jacobians of $\bff(.)$and $\bh(.)$ around $(\bs,\bargamma,\bartheta)$ are
\begin{eqnarray*}
\bF_{\bx}&=&\left[
             \begin{array}{cc}
               0 & 1  \\
               -1-2\bar{\gamma} s_{1}s_{2} &\bar{\gamma}(1-s_{1}^2) \\
             \end{array}
           \right],\label{eq: F}\\
\bH_{\bx}&=&\left[
             \begin{array}{cc}
               \bar{\theta}_{1} & 0  \\
               0 & 0  \\
             \end{array}
           \right],\label{eq: H}\\
\bF_{\gamma}&=&\left[
             \begin{array}{c}
               0 \\
               (1-s_{1}^2)s_{2}  \\
             \end{array}
           \right],\label{eq: R}\\
\bH_{\btheta}&=&\left[
             \begin{array}{cc}
               0&1 \\
	      0&0\\             \end{array}
           \right]\label{eq: P},
\end{eqnarray*}
where $\bs=[s_1~s_2]^T$. Fig. \ref{fig: MLE} depicts the maximum Lyapunov exponent of $\bF_{\bx}-\mu\bH_{\bx}$ as a function of $\mu$, where $\mu$ is the eigenvalue of Laplacian matrix of the network.
\begin{figure}[t]
\begin{center}
\includegraphics[width=3.3in]{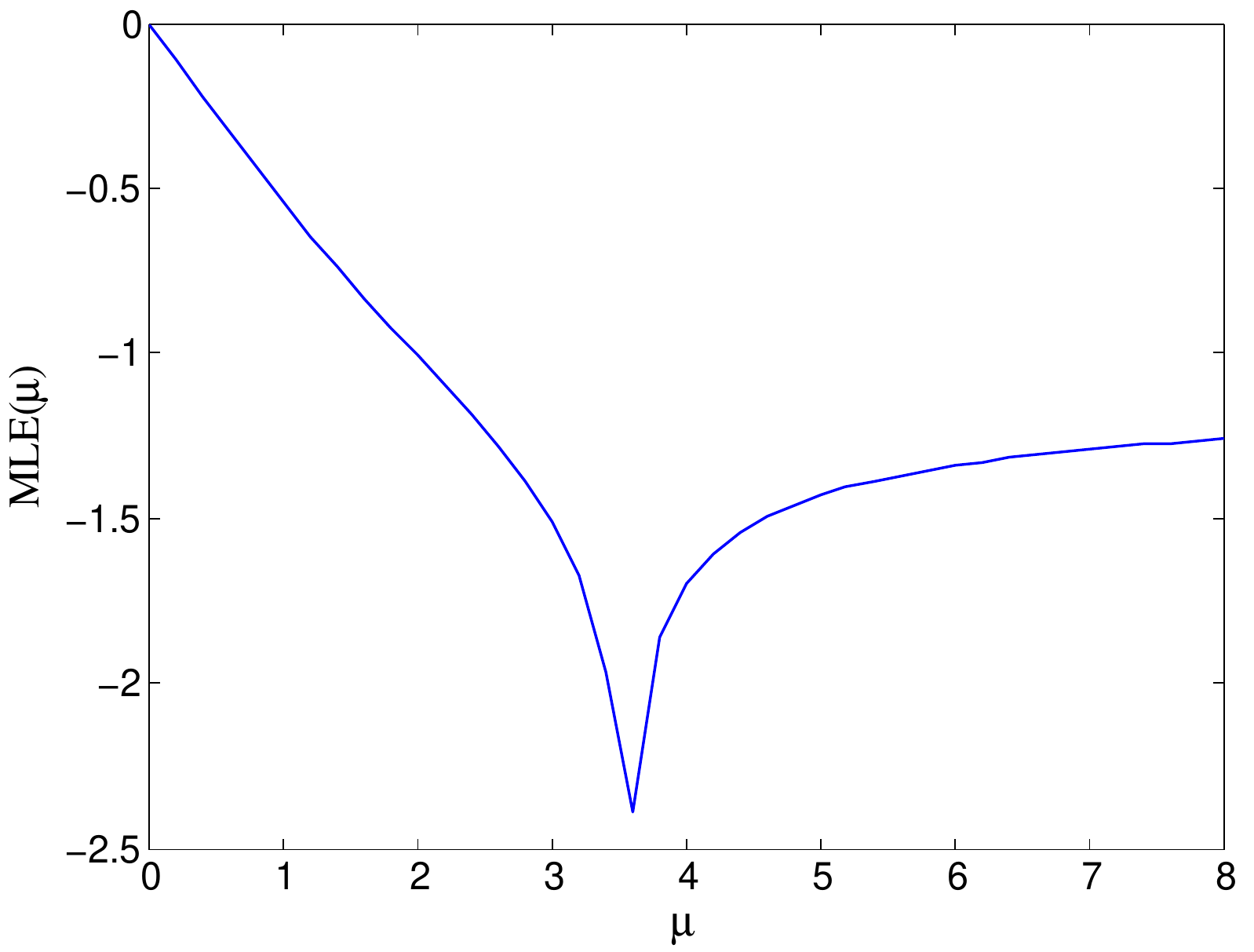}
\caption{Maximum Lyapunov Exponent (MLE) as a function of eigenvalues of Laplacian matrix of the network, $\mu$.}
\label{fig: MLE}
\end{center}
\end{figure}

Furthermore,
\begin{align}
\bv_{i}&=\left[ \begin{array}{c}
\sum_{j=1}^{N}u_{ji}^{*}\sum_{k=1}^{N}a_{jk}\delta \theta_{kj2}\\
 (1-s_{1}^2)s_{2}\sum_{j=1}^{N}u_{ji}^{*}\delta\gamma_{j}\\
 \end{array}\right].\nonumber
 \end{align}
It is clear that $\bv_{i}$ are independent of $\delta\theta_{ij1}$. 
\begin{figure}[t]
\begin{center}
 \includegraphics[width=3.3in]{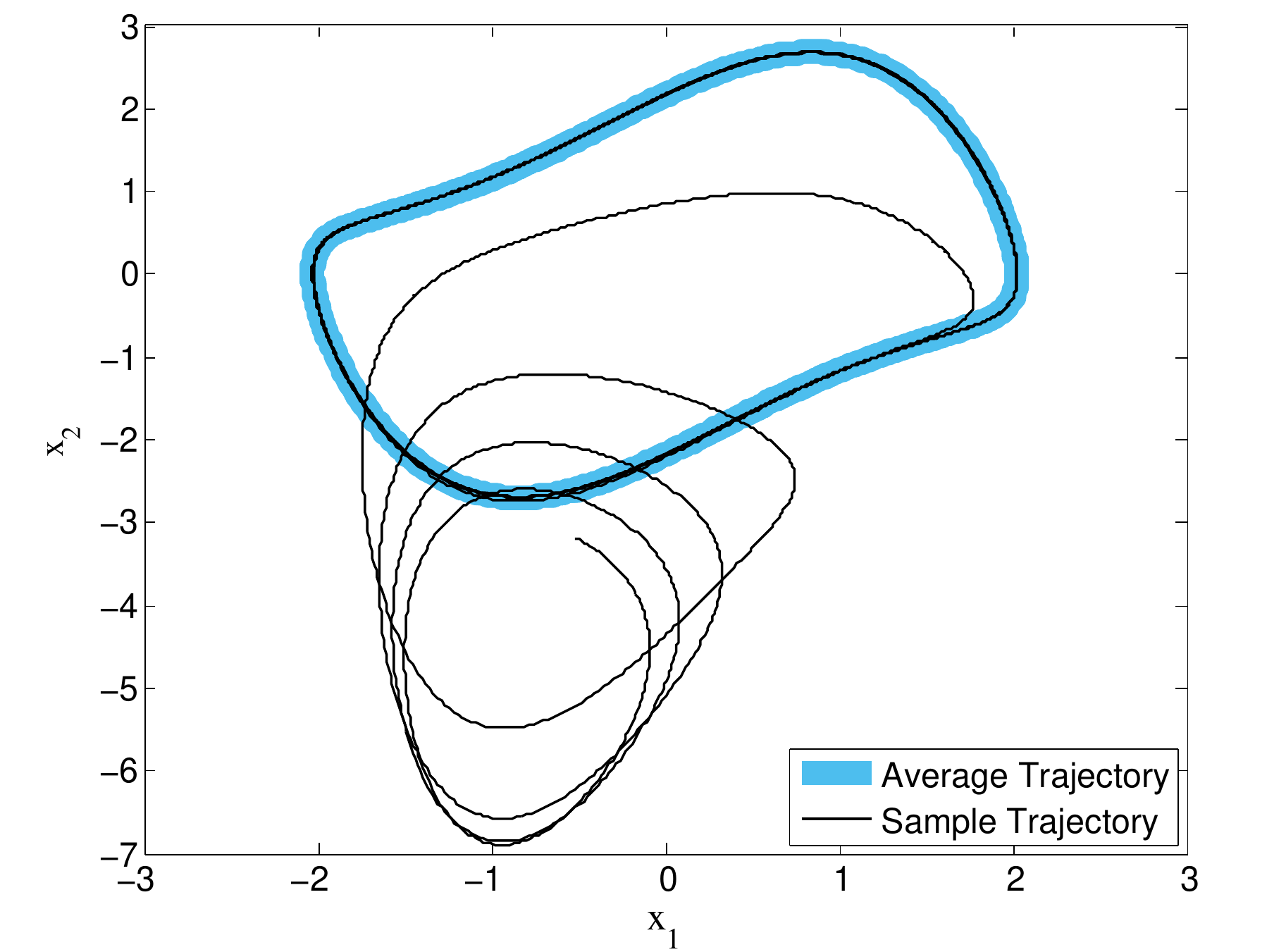}
 \caption{Synchronization manifold, $\bs$, and a sample trajectory, $\bx_1$, for a ring network of van der Pol oscillators.}
 \label{fig: Trajectory}
  \end{center}
\end{figure} 
Also, covariance matrix of $\bv_i$ of a $K$-regular ring network for $\bargamma=1$ can be calculated from \eqref{eq: Cov. Regular} as
\begin{eqnarray}
\bSigma_{ii}=\left[\begin{array}{cc}K{\bar{\theta}}_1^2\sigma_{\theta 2}^2& 0\\0 & \sigma_\gamma^2((1-s_1^2)s_2)^2\end{array}\right],\nonumber
\end{eqnarray}
and
\begin{eqnarray}
\sup \bSigma_{ii}=\left[\begin{array}{cc}K\bar{\theta}_1^2\sigma_{\theta 2}^2& 0\\0 & 9.93\sigma_\gamma^2\end{array}\right],\label{eq:supSigma}
\end{eqnarray}
where $\sup((1-s_1^2)s_2)^2$ is determined by simulation to be $9.93$ and the supremums are calculated over one period of the limit cycle. %Thus, from \eqref{eq: f zeta} 
%\begin{eqnarray}
%f_\zeta(\zeta)=\frac{\zeta}{\sigma^2}e^{-\frac{\zeta^2}{2\sigma^2}},\label{eq: pdf b}\\
%F_\zeta(\zeta)=1-e^{-\frac{\zeta^2}{2\sigma^2}},
%\end{eqnarray}
%where 
Hence, $\sigma=\max(\sqrt{K}\bar{\theta}_1\sigma_{\theta_2},3.15\sigma_\gamma)$.
 
Now, consider a $6$-regular ring network of size $N=100$, where $\gamma\sim\cN(1,0.01)$ and $\btheta\sim\cN([1~0]^T,0.01\bI)$. Fig. \ref{fig: Trajectory} depicts the synchronization manifold, $\bs$, and a sample trajectory, $\bx_1$, converging to $\bs$. Fig. \ref{fig: PstabParameters} presents the analytical lower bound on the probability of stable $\varepsilon$-synchronization in the considered ring network, as a function of $\sigma_{\theta_2}$ and $\sigma_{\gamma}$ for $\varepsilon=0.40$. As it can be seen, the probability of synchronization falls sharply as the variances of mismatches increase. Moreover, we observe that the range of $\sigma_\gamma$ and $\sigma_{\theta_2}$ for which the network is stable with high probability is rectangular. This is explained by noting that $\sigma$ is related to the maximum of $\sigma_\gamma$ and $\sigma_{\theta_2}$, as it can be seen in \eqref{eq:supSigma}. Another observation from Fig. \ref{fig: PstabParameters} is that even small mismatches leads to instability of the synchronization state even with a relatively large tolerance of $\varepsilon=0.40$.

%Fig. \ref{fig: pdf_Ring} shows the histogram of maximum of the norm of the errors from the synchronization manifold, $ \limsup_{t\to\infty}\|\be\|$, in the considered ring network, and compares it to the obtained analytical pdf for its upper bound. The simulations have been performed for 1000 different realizations of mismatch parameters. As expected, the analytical pdf of upper bound on error and simulations are a close match.

We now proceed to compare a ring network, an Erd\"os-R\'enyi network and a Newman-Watts (small-world) network. For a fair comparison, we choose the network parameters such that all networks have the same number of nodes and the same average node degree. That is, we consider a $N=100$ node, $10$-regular ring, an Erd\"os-R\'enyi network with $N=100$ and randomness parameter $p=0.1$, and a Newman-Watts network generated from a $N=100$ node, $6$-regular ring and link addition probability $p=0.4\times 100/94=0.4167$. Fig. \ref{fig: PstabN} presents the probability of stability versus network size, $N$, for these three networks with $\varepsilon=0.4$. As it can be seen for the Ring network (Fig. \ref{fig: PstabN} (a)), as $N$ increases, even though the variance of the mismatch input is constant, $\sigma=3.15\sigma_\gamma$, the $\varepsilon$-synchronization of the network deteriorates. This is because as the degree of the nodes are kept constant and network size increases, the algebraic connectivity\footnote{Algebraic connectivity is defined as the second smallest eigenvalue of the Laplacian matrix of a network.\cite{Mohar91}.} of the network,  
\[\mu^{\scriptsize\mbox{ring}}_{N-1}=k-2\frac{\sin(k\pi/2N)\cos((k+2)\pi/2N)}{\sin(\pi/N)},\]
decreases. For large $N$, in our example, smaller algebraic connectivity means smaller MLE (See Fig. \ref{fig: MLE}), hence the probability of $\varepsilon$-synchronization falls sharply. 

\begin{figure}[t]
\begin{center}
\includegraphics[width=3.3in]{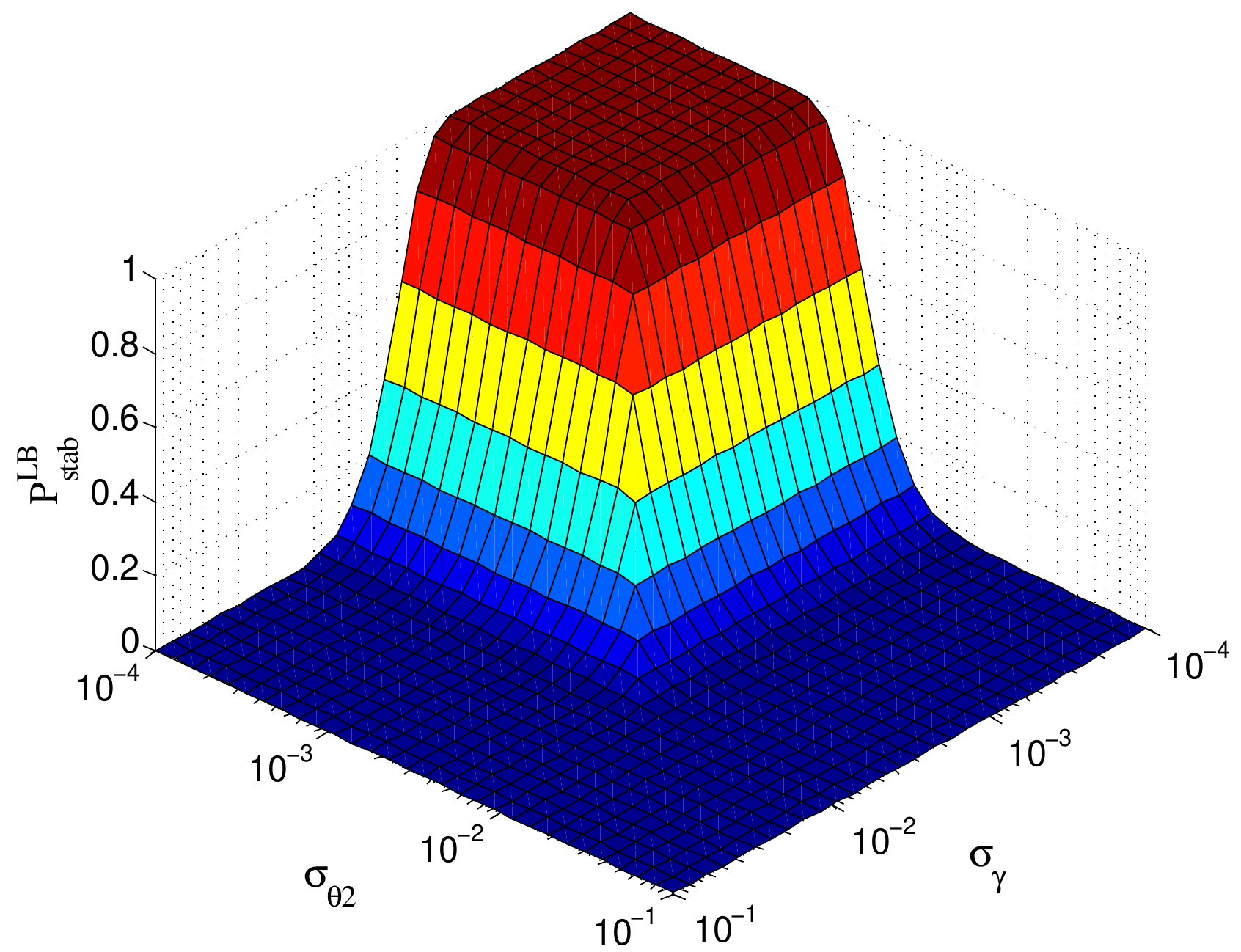}
\caption{$P^{\mbox{\scriptsize LB}}_{\mbox{\scriptsize stab}}$ in the ring network as a function of $\sigma_{\theta_2}$ and $\sigma_{\gamma}$ for $\varepsilon=0.4$.}
\vspace{-0.1in}
\label{fig: PstabParameters}
\end{center}
\end{figure}
Fig. \ref{fig: PstabN} (b) presents the probability of $\varepsilon$-stability of the Erd\"os-R\'enyi network. It is interesting to note that since the network is disconnected for smaller network sizes, the network is not synchronized. As network size continues to grow, the network becomes connected and synchronization behavior emerges. This behavior continues until the growth in the network size, increases the variance of the mismatch input, $\bv$, to the extent that the network falls out of $\varepsilon$-stability.

Fig. \ref{fig: PstabN} (c) presents the probability of $\varepsilon$-stability for the Newman-Watts network. It is interesting to note the mechanisms at work as $N$ increases: At first, when $N$ is small there are very few added links given a small value of $p$. Thus, the network has not yet transitioned into a small-world and its algebraic connectivity is still quite close to that of the ring topology. Thus, as the size of the network increases its second smallest eigenvalue decreases. Since the variance of mismatch, $\bv$, is constant ($\sigma_b=3.15\sigma_\gamma$), the probability of stability decreases. As $N$ continues to increase, by adding links in random, sufficient number of long range connections are established and the small-world transition is achieved. Consequently, algebraic connectivity of the network starts to grow rapidly. Hence, $\lambda_i$ increase and, therefore, $P^{\mbox{\scriptsize LB}}_{\mbox{\scriptsize stab}}$ improves. As $N$ continues to increases $\sqrt{K+Np}\sigma_{\theta_2}$ overtakes $3.15\sigma_\gamma$ in the variance of mismatch and its destructive effect surpasses the improvement caused by transition to small-world. Consequently, we observe that $P^{\mbox{\scriptsize LB}}_{\mbox{\scriptsize stab}}$ begins to drop.

\begin{figure}[!t]
\begin{center}
\begin{tabular}{c}
\includegraphics[width=3.3in]{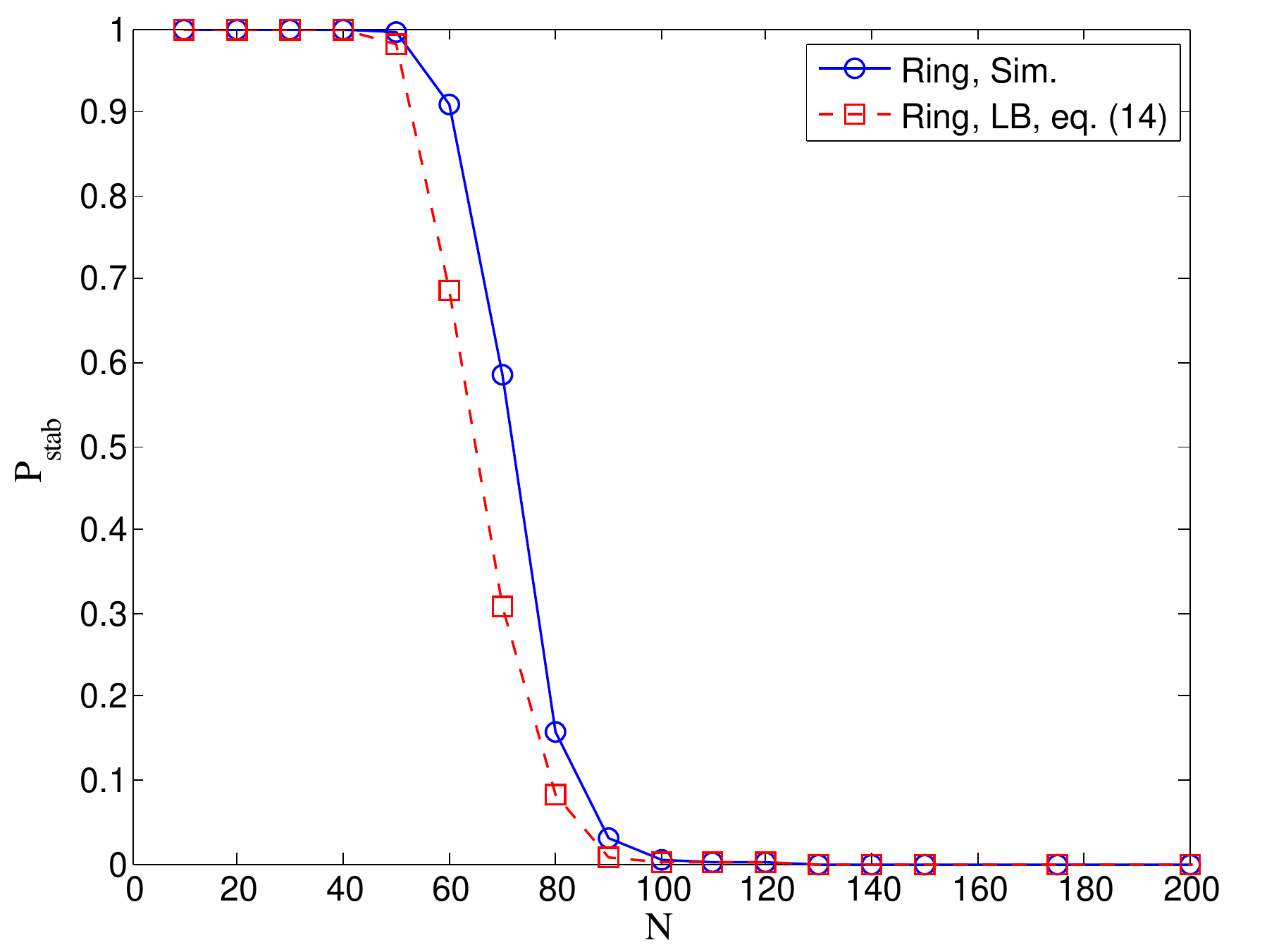}\\(a)\\
\includegraphics[width=3.3in]{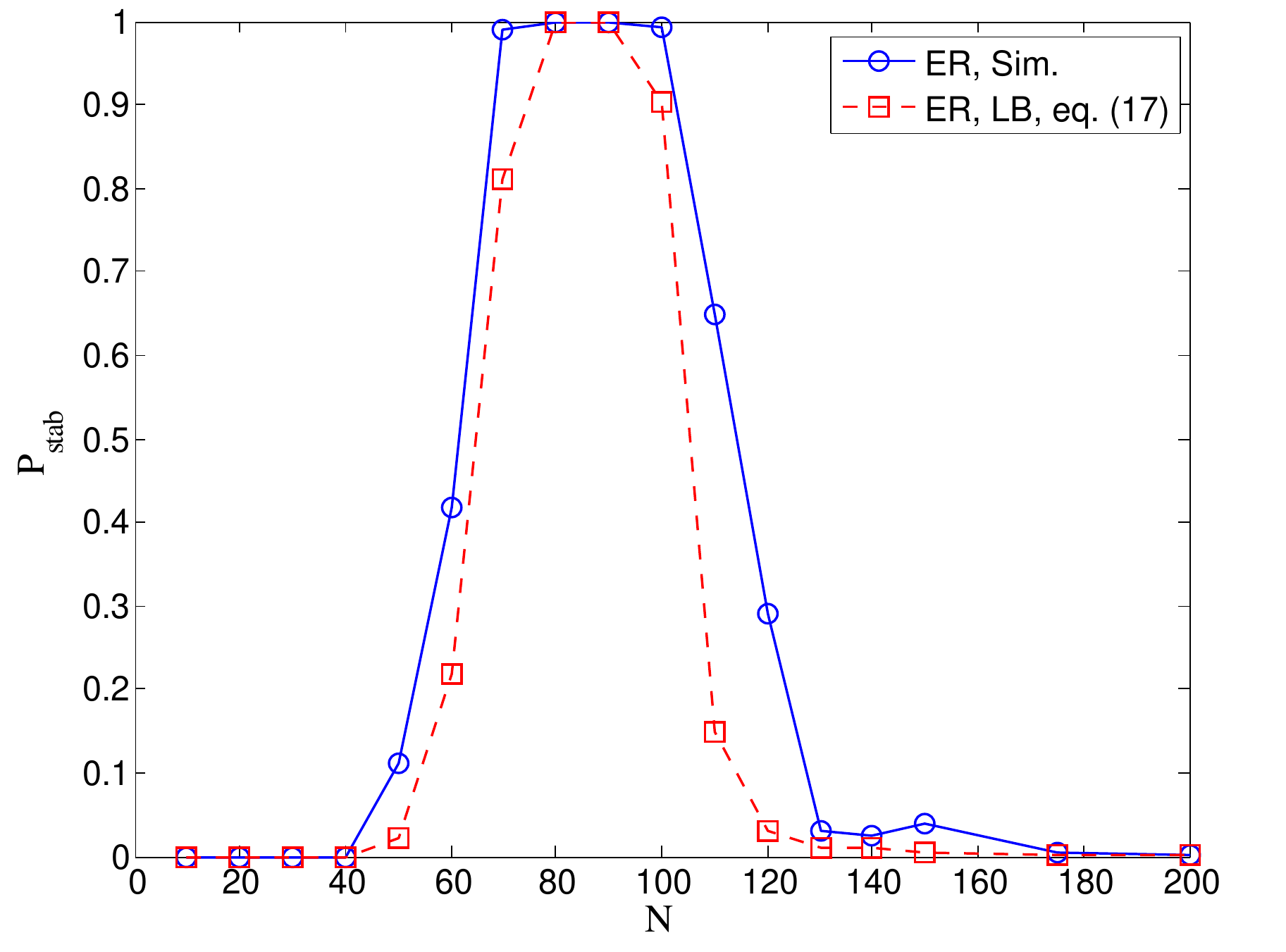}\\(b)\\
\includegraphics[width=3.3in]{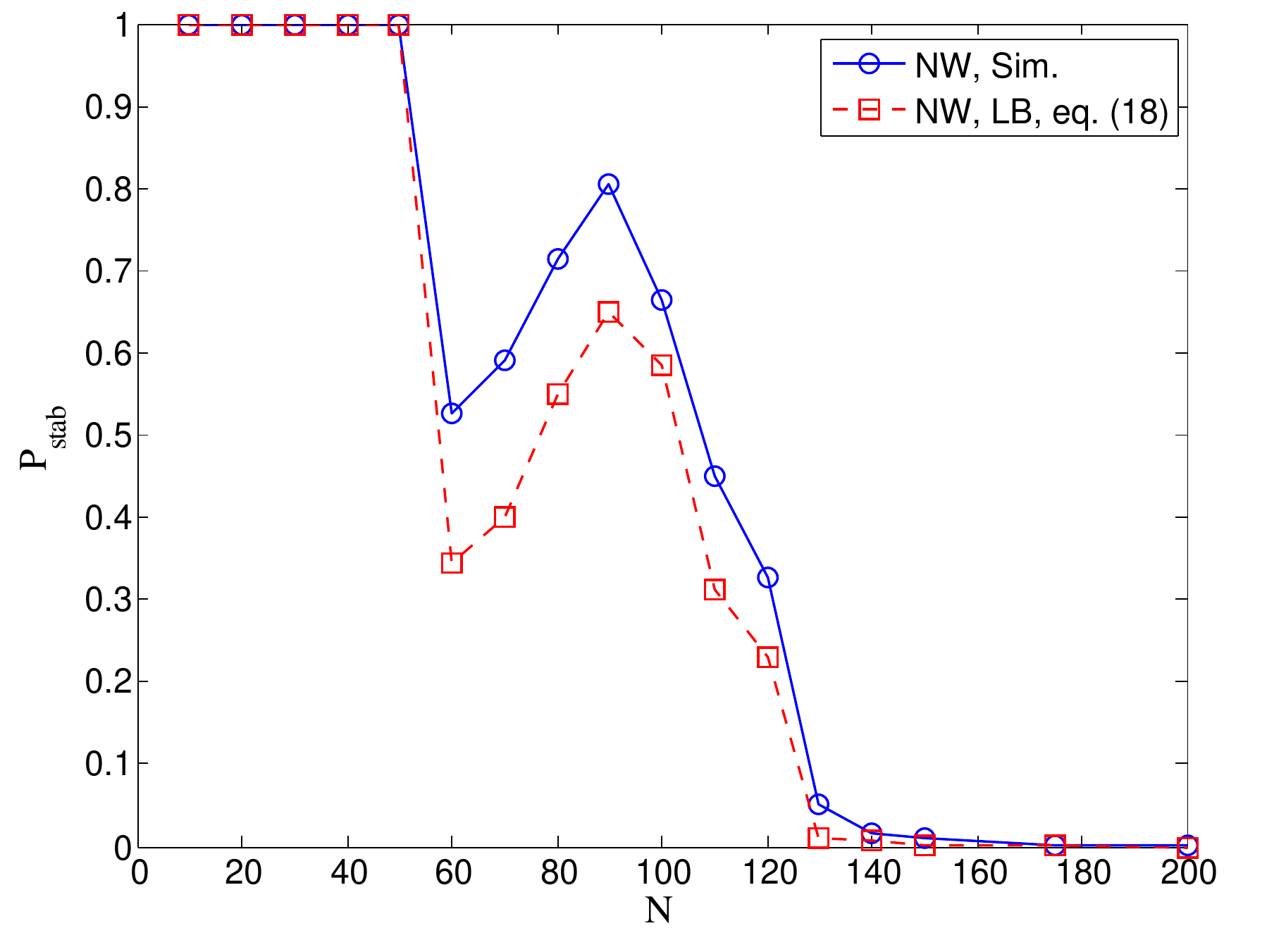}\\(c)
\end{tabular}
\caption{Probability of stability as a function $N$, for (a) ring, (b) Erd\"os-R\'enyi, and (c) Newman-Watts networks.}
\label{fig: PstabN}
\end{center}
\end{figure}

\begin{figure}[!t]
\begin{center}
\begin{tabular}{c}
\includegraphics[width=3.3in]{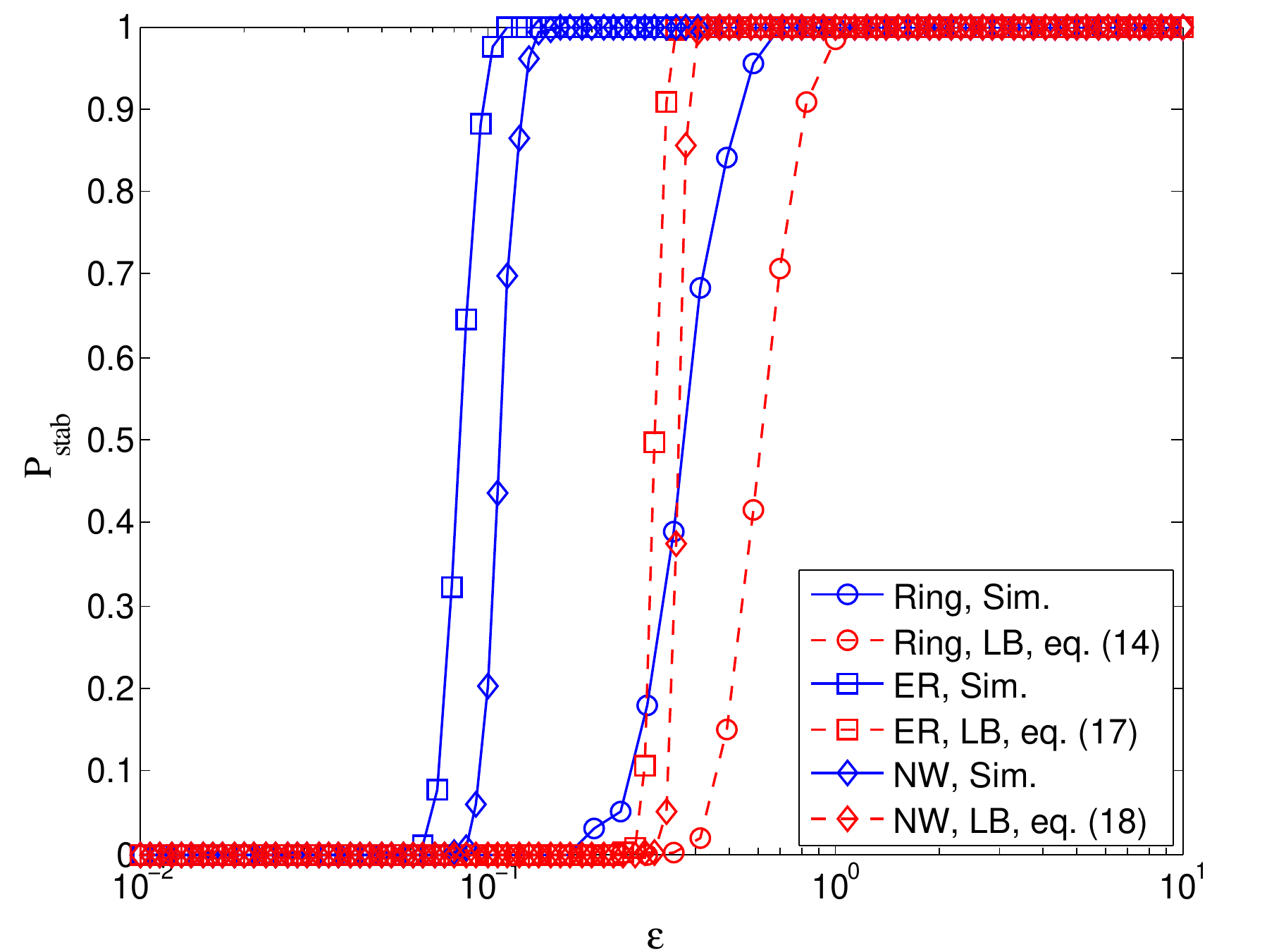}\\(a)\\
\includegraphics[width=3.3in]{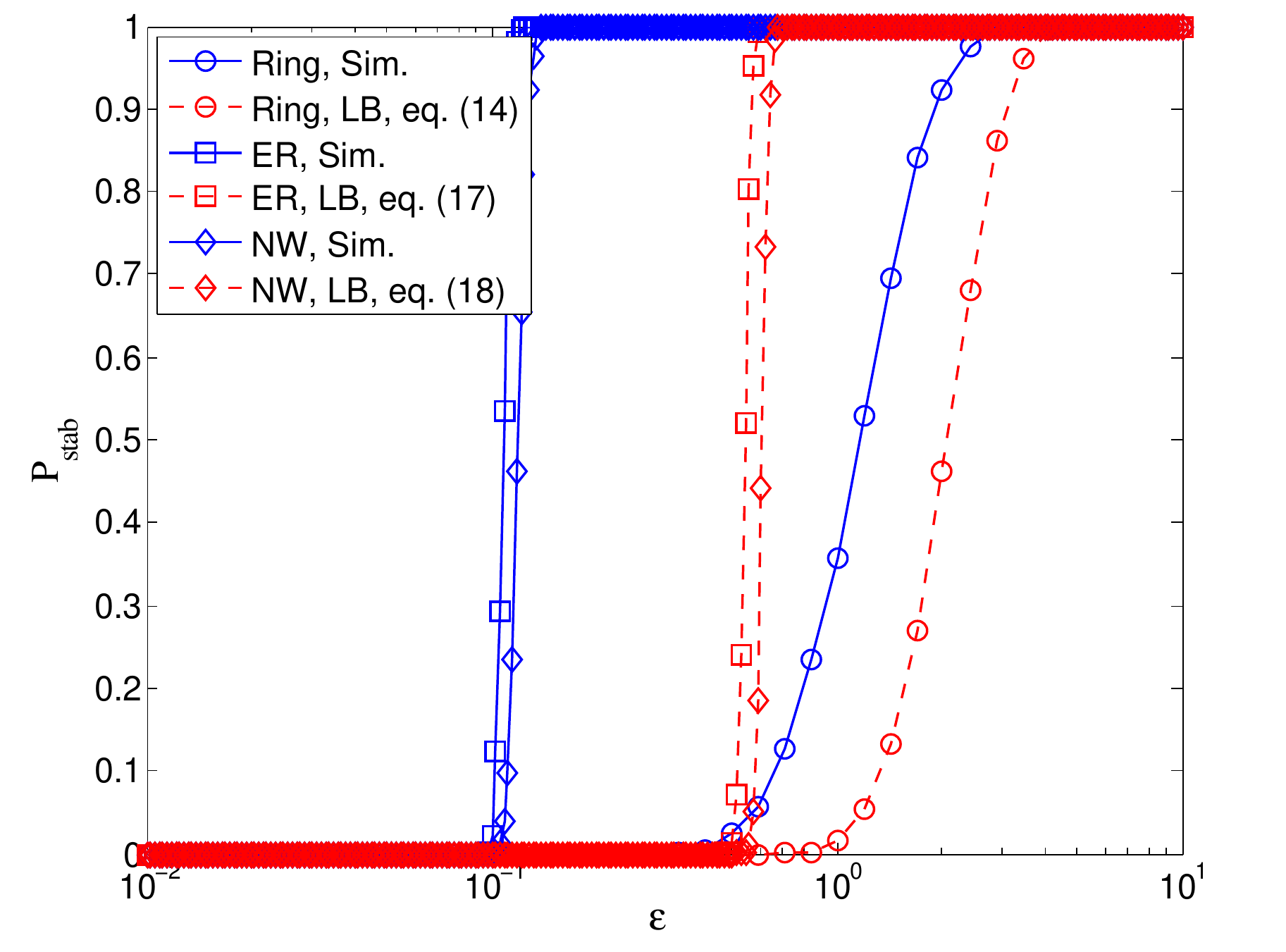}\\(b)\\
\includegraphics[width=3.3in]{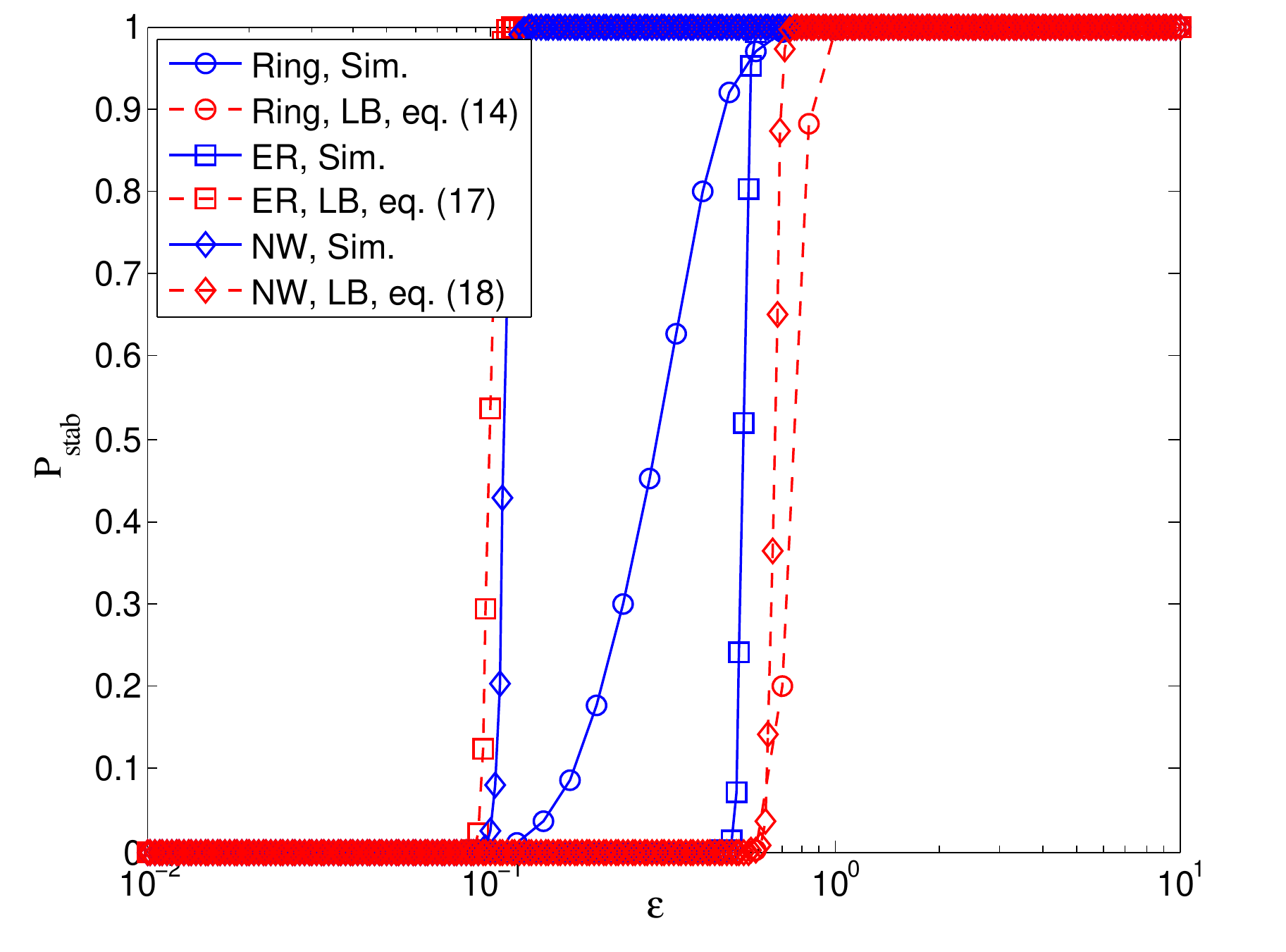}\\(c)
\end{tabular}
\caption{$P_{\mbox{\scriptsize stab}}$ as a function of $\varepsilon$ for the ring, NW and Erd\"{o}s-R\'enyi networks: (a) $N=100, \bar{d}=10$, (b) $N=200, \bar{d}=10$, and $N=200, \bar{d}=20$.}\label{fig: Pe}
\vspace{-0.2in}
\end{center}
\end{figure}
Figs. \ref{fig: Pe} (a) through (c) depict the probability of $\varepsilon$-stability as a function of $\varepsilon$ in the considered Ring, Erd\"os-R\'eyni, and Newman-Watts networks for different $N$ and $\bar{d}$: (a) $N=100, \bar{d}=10$, (b) $N=200, \bar{d}=10$, and $N=200, \bar{d}=20$. As it can be seen, the analytical lower bound and the simulation result for the ring network are reasonably close. This is due to the homogeneity of its node degrees, i.e. $d_i=K$, which holds true for the other networks as $N$ approaches infinity. The other point directly observed from these figures is that the rise in the probability of the stability is much sharper in the Erd\"os-R\'eyni and Newman-Watts networks, this is because the spread of the spectrum, [$\mu_{\min},\mu_{\max}$], for these networks are smaller than that of ring topology. This, in fact, causes the Lyapunov exponents of the traverse modes to be closer to each other and hence the networks become easily and rapidly synchronized. Other interesting observation is that the results for the Erd\"os-R\'eyni and Newman-Watts networks are similar. The reason can be sought in the effectiveness of communication in both networks to each other. As it has been shown in \cite{Watts98}, even though small-worlds are strongly locally connected (due to ring substrate), they have almost the same average shortest path length of Erd\"os-R\'eyni networks. This results in almost the same communication efficiency in small-worlds as Erd\"os-R\'eyni network. Hence, the synchronizability of both types of networks are similar.

%We see that the analytical results closely match those of simulations. Moreover, as the network's algebraic connectivity increases (for larger $p$), which means that network becomes {\em smaller}, the network tends to become synchronized with smaller $\varepsilon$. This is due to the fact that the nodes have higher maximum Lyapunov exponent for networks with higher algebraic connectivity. This can also be seen in the obtained generalized MSF in \eqref{eq: theo1}, which suggests that larger maximum Lyapunov exponent results in smaller $\varepsilon$ and faster convergence. 

%% -------------------------SECTION (VI)        CONCLUSION-------------------
\section{Conclusion}
We had seen that mismatch in either couplings or the local dynamics does not allow perfect synchronization. Rather, the network can only be synchronized to a neighborhood of the synchronization manifold. Considering this relaxed notion of synchronization we have provided a generalized master stability function that takes the mismatches into account. We then used this master stability function to derive lower bounds on the probability of synchronization in regular, Erd\"os-R\'enyi, and Newman-Watts networks. We verified our results using numerical examples involving networks of van der Pol oscillators. These examples clearly shows the different phase transition behavior of the different network models.

\appendices
\section{Proof of Lemma 2}\label{NonDiag}
\begin{proof}
The state space equation for the $\bbeta_i$ can be written as
\begin{eqnarray}\label{eq: NodeDiag1}
\dot{\bbeta_{i}}&=&
\left(\bF_{\bx}-\mu \bH_{\bx}\right)\bbeta_{i}+(\bF_{\bx}-\bH_{\bx} )\bbeta_{i+1}+\bv_i(t),
\end{eqnarray}
for $i\ne N$, and
\begin{eqnarray}\label{eq: NodeDiag2}
\dot{\bbeta_N}&=&
\left(\bF_{\bx}-\mu \bH_{\bx}\right)\bbeta_N+\bv_N(t).
\end{eqnarray}
The solution of \eqref{eq: NodeDiag1} and \eqref{eq: NodeDiag2} are 
\begin{eqnarray}\label{eq: Solution}
\bbeta_{i}(t)&=&\mathbf{\Phi}(t,0)\bbeta_{i}(0)+\int_{0}^{t}\mathbf{\Phi}(t,\tau)\bv_{i}(\tau)d\tau\nonumber\\
&&+\int_{0}^{t}\mathbf{\Phi}(t,\tau)(\bF_{\bx}-\bH_{\bx})\bbeta_{i+1}(\tau)d\tau, \nonumber
\end{eqnarray}
for $i\ne N$ and
\begin{eqnarray}\label{eq: Solution}
\bbeta_N(t)&=&\mathbf{\Phi}(t,0)\bbeta_N(0)+\int_{0}^{t}\mathbf{\Phi}(t,\tau)\bv_N(\tau)d\tau,\nonumber
\end{eqnarray}
where $\mathbf{\Phi}(t,\tau)=\mathbf{Z}(t)\mathbf{Z}^{-1}(\tau)$, and $\mathbf{Z}$ is the normal fundamental matrix of $\bF_{\bx}-\mu\bH_{\bx}$ \cite{Daleckii74}.

Applying triangle inequality yields
\begin{eqnarray}
\lVert\bbeta_{i}(t)\rVert &\leq& \lVert\mathbf{\Phi}(t,0)\rVert \lVert\bbeta_{i}(0)\rVert\nonumber\\
&&+ \int_{0}^{t}\lVert(\bF_{\bx}-\bH_{\bx})\bbeta_{i+1}(\tau)
+\bv_{i}(\tau)\rVert\lVert\mathbf{\Phi}(t,\tau)\rVert d\tau,\nonumber
\end{eqnarray}
for $i\in\{1,...,N-1\}$, and
\begin{eqnarray}
\lVert\bbeta_{N}(t)\rVert &\leq&\lVert\mathbf{\Phi}(t,0)\rVert \lVert\bbeta_N(0)\rVert +  \int_{0}^{t}\lVert\bv_N(\tau)\rVert\lVert\mathbf{\Phi}(t,\tau)\rVert d\tau,\nonumber
\end{eqnarray}
which, as $t\to\infty$, yields
\begin{eqnarray}
\limsup_{t\to\infty}\lVert\bbeta_{i}\rVert & \leq & \lVert\bbeta_{i}(0)\rVert\limsup_{t\to\infty}\lVert\mathbf{\Phi}(t,0)\rVert\nonumber\\
&&+\limsup_{t\to\infty}\lVert\bv_{i}\rVert \limsup_{t\to\infty}\int_{0}^{t}\lVert\mathbf{\Phi}(t,\tau))\rVert d\tau\nonumber\\
&&+\limsup_{t\to\infty}\lVert\bF_{\bx}-\bH_{\bx}\rVert\limsup_{t\to\infty}\lVert\bbeta_{i+1}\rVert\nonumber\\
&&\quad\quad\quad\quad\times\limsup_{t\to\infty}\int_{0}^{t}\lVert\mathbf{\Phi}(t,\tau)\rVert d\tau,\label{Norm Error1}
\end{eqnarray}
and
\begin{align}
\limsup_{t\to\infty}\lVert\bbeta_{N}\rVert \leq & \lVert\bbeta_N(0)\rVert \limsup_{t\to\infty}\lVert\mathbf{\Phi}(t,0)\rVert \nonumber\\
&+\limsup_{t\to\infty}\lVert\bv_{N}\rVert\limsup_{t\to\infty}\int_{0}^{t}\lVert\mathbf{\Phi}(t,\tau)\rVert d\tau.\label{Norm Error2}
\end{align}
We know that there exists positive real $\phi$ such that \cite{Daleckii74}
\begin{align*}
\lVert\mathbf{\Phi}(t,\tau)\rVert \leq \phi e^{-\lambda(t-\tau)},
\end{align*}
where $\lambda$ is the maximum Lyapunov exponent of $\bF_{\bx}-\mu\bH_{\bx}$. If $\lambda>0$ this yields
\begin{eqnarray}
\limsup_{t\to\infty}\lVert\mathbf{\Phi}(t,0))\rVert & = & 0,\nonumber \\
\limsup_{t\to\infty}\int_{0}^{t}\lVert\mathbf{\Phi}(t,\tau))\rVert d\tau & \le & \frac{\phi}{\lambda}.\nonumber
\end{eqnarray}
Substituting in \eqref{Norm Error1} and  \eqref{Norm Error2} yields, 
\begin{eqnarray}
\limsup_{t\to\infty}\lVert\bbeta_{i}\rVert & \leq & \frac{\phi}{\lambda} \limsup_{t\to\infty}\lVert\bF_{\bx}-\bH_{\bx}\rVert\lVert\bbeta_{i+1}\rVert\nonumber\\
&&+\frac{\phi}{\lambda}\limsup_{t\to\infty}\lVert\bv_{i}\rVert,\nonumber
\end{eqnarray}
and
\begin{eqnarray}
\limsup_{t\to\infty}\lVert\bbeta_{N}\rVert & \leq & \frac{\phi}{\lambda} \limsup_{t\to\infty}\lVert\bv_{N}\rVert.\nonumber
\end{eqnarray}
Solving the recursive inequalities we get
\begin{eqnarray}
\limsup_{t\to\infty}\lVert\bbeta_{i}\rVert & \leq & \sum_{j=i}^{N} \left(\frac{\phi}{\lambda}\right)^{N-j+1}\limsup_{t\to\infty}\lVert\bF_{\bx}+\bH_{\bx}\rVert^{N-j}\nonumber\\
&&\quad\quad\quad\quad\times\limsup_{t\to\infty}\lVert\bv_{j}\rVert.\nonumber
\end{eqnarray}
\end{proof}

%%%%%%%%%%%%
%%%%%%%%%%%%%
\section{Covariance of $\bv$}\label{covariance}
The covariance of $\bv$ is 
\begin{align}\label{correlation1}
\bSigma_{\bv}={E}[(\bv-\bar{\bv})(\bv-\bar{\bv})^{T}]= [\bSigma_{ij}],\nonumber
\end{align}
where $\bar{\bv}=\bh(\bs,\bs,\bar{\btheta})\otimes(\bD^{\scriptsize\mbox{in}}-\bar{d}_{\scriptsize\mbox{in}}\bI)$ and $\bD^{\scriptsize\mbox{in}}=\mbox{diag}([d_{1}^{\scriptsize\mbox{in}} \cdots d_{N}^{\scriptsize\mbox{in}}])$. Thus, the $ij$th block of $\bSigma_{\bv}$ is
%\begin{figure*}[t]
\begin{eqnarray}
\bSigma_{ij}&=&\sum_{k,l\in\cN}u_{ik}u_{jl}^{\star}E\left[\left(\bF_{\bgamma}\delta\bgamma_{k}+\sum_{m\in\cN}a_{mk}\bH_{\btheta}\delta\btheta_{mk}\right)\right.\nonumber\\
&&\times \left.\left(\bF_{\bgamma}\delta\bgamma_{l}+\sum_{n\in\cN}a_{ln}\bH_{\btheta}\delta\btheta_{nl}\right)^{T}\right]\nonumber\\
&=&\sum_{k,l\in\cN}u_{ik}u_{jl}^{\star}\bF_{\bgamma}\bSigma_{\bgamma}\bF_{\bgamma}^{T}\nonumber\\
&&+\sum_{k,l,m,n\in\cN}u_{ik}u_{jl}^{\star}a_{km}a_{ln}\bH_{\btheta}\bSigma_{\btheta}\bH_{\btheta}^{T}\mathbbm{1}_{k=l}\mathbbm{1}_{m=n}\nonumber\\
&=&\bF_{\bgamma}\bSigma_{\bgamma}\bF_{\bgamma}^{T}1_{i=j}+\sum_{k,m\in\cN}u_{ik}u_{jk}^{\star}\lvert a_{km}\rvert^{2}\bH_{\btheta}\bSigma_{\btheta}\bH_{\btheta}^{T}.\nonumber
\end{eqnarray}

\end{document}